\documentclass[11pt]{amsart}
\usepackage{amsmath, amsthm, amssymb, hyperref, color}
\usepackage{xcolor}
\hypersetup{
    colorlinks,
    linkcolor={red!80!black},
    citecolor={blue!80!black},
    urlcolor={blue!80!black}
}
\usepackage{graphicx}
\usepackage{cleveref}
\usepackage{mathtools}
\usepackage{enumerate}
\usepackage{stackengine}
\usepackage{verbatim}
\usepackage{tikz}
\usepackage{tikz-cd}
\newtheorem{theorem}{Theorem}[section]
\newtheorem{proposition}[theorem]{Proposition}
\newtheorem{lemma}[theorem]{Lemma}
\newtheorem{corollary}[theorem]{Corollary}
\theoremstyle{definition}
\newtheorem{definition}[theorem]{Definition}

\newtheorem{remark}[theorem]{Remark}

\newtheorem{question}[theorem]{Question}

\newcommand{\PP}{\mathbb{P}}

\newcommand{\RR}{\mathbb{R}}

\newcommand{\CC}{\mathbb{C} }
\newcommand{\ZZ}{\mathbb{Z}}

\newcommand{\BB}{\mathbb{B}}
\newcommand{\HH}{\mathbb{H}}
\newcommand{\fg}{{\mathfrak{g}}}
\newcommand{\cO}{{\mathcal{O}}}
\newcommand{\cJ}{{\mathcal{J}}}
\newcommand{\cV}{{\mathcal{V}}}
\newcommand{\gth}{\theta}
\newcommand{\ga}{\alpha}
\newcommand{\gb}{\beta}
\newcommand{\gd}{\delta}
\newcommand{\gD}{\Delta}
\newcommand{\gep}{\epsilon}
\newcommand{\gS}{\Sigma}

\newcommand{\Sym}{\mathrm{Sym}}
\newcommand{\Tail}{\mathrm{Tail}}
\newcommand{\Term}{\mathrm{Term}}
\newcommand{\codim}{\mathrm{codim}}
\newcommand{\rmspan}{\mathrm{span}}
\newcommand{\ol}[1]{\overline{#1}}
\newcommand{\ssm}{\smallsetminus}
\usepackage{color}

\begin{document}

\title[Reconstruction of curves from theta hyperplanes in genera $6, 7$]{Reconstruction of curves from their theta hyperplanes in genera $6$ and $7$}
\author{T\"urk\"u \"Ozl\"um \c{C}el{\.I}k}
\address{Max Planck Institute of Molecular Cell Biology and Genetics, Dresden}
\email{celik@mpi-cbg.de}
\author{David Lehavi}
\address{Tel Aviv}
\email{dlehavi@gmail.com}
\date{}
\noindent \begin{abstract}
We derive a formula for reconstructing a generic complex canonical curve $C$ of genus 6 and 7 in terms of the theta hyperplanes of $C$. Hence, we get a generic inverse to the Torelli map, as well as a complete description of the Schottky locus in these genera. The computational part of the proof relies on a certified numerical argument.\\
\end{abstract}
\maketitle
%
\section{Introduction}\label{sec:intro}
%
A generic complex algebraic curve $C$ of genus $g$ possesses $2^{g-1}(2^g-1)$
half canonical divisors $\gth_i$ with $\dim H^0(C,\gth_i)=1$ which are called the
{\em odd theta characteristics} of $C$.
These theta characteristics uniquely determine the curve, as shown
in~\cite{CapSer}. By identifying the curve with its canonical image in
$|K_C|^*\cong\PP^{g-1}$, one can consider the \emph{theta hyperplanes} as the
corresponding geometric objects on the canonical model, which meet the curve in
the projective space at $g-1$ points with multiplicity 2 at each point. This
classical topic in algebraic geometry can be studied from various perspectives.
One interesting approach is to reconstruct the canonical image of the curve
given all the theta hyperplanes. This problem has been extensively investigated
for non-hyperelliptic curves of genus 3, with pioneering work by
Aronhold~\cite{Aro} and Coble~\cite{Cob}. For recent expositions see
\cite[6.1.2]{Dol} and \Cref{rmr:n1} respectively. For curves of higher genera,
explicit constructions have been studied for genus $4$
in~\cite{LehGenus4, CelKulRenNamGenus4},
and for genus $5$ in~\cite{LehGenus5}. In this
article, we extend these results to curves of genera $6$ and $7$, providing new
insights into the general reconstruction problem.

In \Cref{sec:main} we outline the proof of our main result, namely
\Cref{thm:main}. It explicitly describes generic complex canonical curves of
genera 6 and 7 in terms of their theta hyperplanes. \Cref{cor:Schottky} provides
a solution to the Schottky problem in genera 6 and 7, which asks for a
characterization of the locus of Jacobian varieties among principally polarized
Abelian varieties. As in previous works~\cite{LehGenus4,LehGenus5}, we start
by proving the claim for a specific curve in each genus,
and conclude by using a classical degeneration argument.
In the current work these specific curves are Wiman's sextic and
the Fricke-Macbeath curve for genera $6$ and $7$ respectively.

The proof of our reconstruction formula for these two curves
relies on a combination of symbolic and
non-certified numerical algorithms. The latter produces non-certified
analytical representations of the hyperplanes. Thus, most of our work is to
certify the non-certified data.
In \Cref{sec:certificationThetaHyperplanes} we build a certification tool for
the theta hyperplanes, which may be of independent interest in both computing
and certifying intersection points acquiring multiplicities.
In \Cref{sec:certificationSteiner} we show how to certify the level 2 structure
on the theta hyperplanes,
and in \Cref{sec:dimcert} we show how to certify the reconstruction formula
itself. Finally, in \Cref{sec:program} we conclude the proof of
the main theorem by presenting the measured bounds on the input accuracies and
the computed accuracy bounds attained by the methods of the previous sections.
The paper contains two appendices in which we discuss classical results:
Riemann matrices in
\Cref{sec:computingThetaHyerplanes} and numerical linear algebra in
\Cref{sec:linear}.
%
\section{Notations, statement of the result, and outline of the proof}\label{sec:main}
%
We start by establishing our notation, which will be used throughout this paper
unless otherwise specified. Assume that $C$ is a generic complex curve of genus
$g$, and let $K_C$ denote the canonical divisor of $C$. As $C$ is generic, an
\emph{odd theta characteristic} $\theta$ of $C$ is a half-canonical divisor,
$2\theta=K_C$, with ${\dim}H^0(C,\theta)=1$. By considering the canonical model
of the curve in $|K_C|^*\cong\PP^{g-1}$, we can identify $\theta$ with a unique
hyperplane $H_\theta\subset |K_C|^*$ such that the intersection
$H_\theta\cdot C=2\theta$, where $C$ also represents its canonical model. We call
$H_\theta$ the \emph{theta hyperplane} arising from $\theta$. We denote the
homogeneous linear form in $g$ variables that defines $H_\theta$ by $l_\theta$.

An important object encapsulating relations between theta hyperplanes is the
\emph{Steiner set}. Denote the group of 2-torsion points of the Jacobian of $C$
by $JC[2]$. Let $0\neq \ga\in JC[2]$, then the Steiner set $\gS_\ga$
is composed of pairs ${\theta,\theta+\ga}$ with both being odd. Symbolically, we
have
\begin{equation}\label{def:SteinerSet}
\gS_\ga = \{\{\theta,\theta+\ga\}: {\dim}H^0(C,\theta)={\dim}H^0(C,\theta+\ga)=1\}.
\end{equation}
See \cite[5.4.2]{Dol} for further discussion.
For such a pair $\theta, \theta+\alpha$, we let the quadratic form
$q_{\{\gth,\gth+\ga\}}\in |\cO_{|K_C|^*}(2)|$ be the
image of $\{[H_\gth], [H_{\gth+\ga}]\}\in \Sym^2|K_C|$ under the map
$\Sym^2|K_C|\to |\cO_{|K_C|^*}(2)|$, namely the projectivization of the product of
the linear forms $l_\gth\cdot l_{\gth+\ga}$. We recall some notation that will
furnish our objects of study. We denote by $\gD$ the diagonal map:
$\gD:H^0(K_C+\ga)\to \Sym^2H^0(K_C+\ga)$.
Let $\nu_n$ be the Veronese map from $\Sym^2H^0(K_C+\ga)$ to
$\Sym^n\Sym^2H^0(K_C+\ga)$. We also consider the canonically
determined maps
$\pi_{K_C}:\Sym^2H^0(K_C)\to H^0(2K_C)$ and $\pi_{K_C+\ga}:\Sym^2H^0(K_C+\ga)\to H^0(2K_C)$. Last but not least, we will
use the $n$-th symmetric powers of the maps $\pi_{K_C},\pi_{K_C+\ga}$, which we
denote by $\pi^n_{K_C},\pi^n_{K_C+\ga}$.
\begin{definition}
  Let $U_{n,\ga}:=\pi^n_{K_C+\ga}\circ \nu_n\circ\gD(H^0(K_C+\ga))$, and let
  $V_{n,\ga}$ be linear span of $(\pi^n_{K_C})^{-1}U_{n,\ga}$.
\end{definition}
The maps and the objects defined above are related as in the following diagram,
whose essence is to relate multi-linear algebra on the Prym canonical system
$H^0(K_C+\ga)$ and on the canonical system.
\[
    \begin{tikzcd}[row sep=tiny, column sep=small]
      H^0(K_C+\ga) \ar[r, "\Delta"] & \Sym^2H^0(K_C+\ga) \ar[r, "\nu_n"]
      & \Sym^n\Sym^2H^0(K_C+\ga) \ar[rd, "\pi^n_{K_C+\ga}"]\\
      &&& \Sym^nH^0(2K_C) \supset U_{n,\ga}.\\
      && V_{n,\ga} \subset \Sym^n\Sym^2H^0(K_C) \ar[ur, "\pi^n_{K_C}"]
\end{tikzcd}
\]
Finally, let $I_n(\gD (X))$ be the ideal of degree $n$ forms vanishing on the
image $\gD (X)\subset \Sym^2X$ of $X$ under the diagonal map inside $\Sym^2X$.
This is the ideal corresponding to the Veronese image of this object, namely
$\nu_n(\gD(X))$. We are interested in the case when $X=H^0(K_C+\ga)$.
We now introduce our main object of study.
\begin{definition}\label{dfn:PropertiesABC}
  We say that the curve $C$ has {\em properties} {\bf A}, {\bf B}, {\bf C} with
  respect to $n$  if the following statements hold respectively.
\begin{itemize}
    \item[({\bf A})] The codimension of $V_{n,\ga}$ is
    \begin{equation}\label{codimV}
     d_\ga := \dim I_n(\gD(H^0(K_C+\ga)) - \dim \Sym^n\Sym^2H^0(K_C+\ga) + \dim \Sym^nH^0(2K_C).
    \end{equation}
  \item[({\bf B})] The linear space $V_{n,\ga}$ is spanned by the images of
    $q_{\{\gth,\gth+\ga\}}$.
    \item[({\bf C})] The intersection $\cap_{\ga\in JC[2]\ssm\{0\}}V_{n,\ga}$ is
      the kernel of $\pi^n_{K_C}$.
\end{itemize}
\end{definition}
Loosely speaking, properties {\bf B},{\bf C} determine whether the pairs of
theta hyperplanes are in \emph{general linear position}, up to known dimension
constraints. This makes these claims amenable to both specialization and
computational techniques. We make the first of these two observations precise
below.
\begin{lemma}\label{lem:codimV}
The codimension of $V_{n,\ga}$ is at least $d_\ga$.
\end{lemma}
\begin{proof} We have the following equality by definition
\[
    \dim \rmspan\ \nu_n\circ\gD(H^0(K_C+\ga))=\dim \Sym^n\Sym^2H^0(K_C+\ga) - \dim I_n(\gD(H^0(K_C+\ga)).
\]
This implies that
\[
\dim \pi_{K_C+\ga}\rmspan\ \nu_n\circ\gD(H^0(K_C+\ga)) \leq
\dim \Sym^n\Sym^2H^0(K_C+\ga) - \dim I_n(\gD(H^0(K_C+\ga)).
\]
Then, by flipping the inequality, we have
\begin{align*}
  \codim \pi_{K_C+\ga}\rmspan\ & \nu_n\circ\gD(H^0(K_C+\ga)) \geq \\
& \dim \Sym^n H^0(2K_C)-\left(\dim \Sym^n\Sym^2H^0(K_C+\ga) - \dim I_n(\gD(H^0(K_C+\ga))\right).
\end{align*}
As the pull back preserves the codimension, the following codimensions are the same:
\[
   \codim{(\pi^n_{K_C}})^{-1}\pi_{K_C+\ga}\rmspan\ \nu_n\circ\gD(H^0(K_C+\ga)) =
\codim\pi_{K_C+\ga}\rmspan \nu_n\circ\gD(H^0(K_C+\ga)),
\]
which concludes the proof.
\end{proof}
Property {\bf B} is a general position property since it describes when a finite
number of points in a space spans this space. Property {\bf C} can also be seen
as a general position property: Given some spaces $V_{n,\ga}$ that are strictly
contained in $\Sym^n\Sym^2H^0(K_C)$ and are all containing some subspace $V$,
the property asserts that intersection $\cap_{\ga\in JC[2]\ssm\{0\}}V_{n,\ga}$ is
merely $V$.
\begin{remark}[$n = 1$]\label{rmr:n1}
  In this case, we have $\dim I_n(\gD(H^0(K_C+\ga))=0$, and so the
  codimension bound $d_\ga$ in Property {\bf A} becomes
$\dim H^0(2K_C)-\dim \Sym^2H^0(K_C+\ga)=3g-3 - g(g-1)/2$,
which is non-positive when $g>5$. This agrees with the map
$\pi_{K_C+\ga}$ being surjective, which has been known for genus $g>6$ since
\cite{LaSe}, and
the case of genus 6 is the first case of the Prym-Green conjecture. This case
was settled in \cite{ChiEisFarSch}.

The content of Properties {\bf A}, {\bf B}, {\bf C} with $n=1$ and generic genus $3$ curves
is essentially
classical: Indeed, Coble knew that Properties {\bf A}, {\bf B} hold for generic curves
\cite[Chapter IV Section 47 (6)]{Cob},
\cite[Chapter VI exercises F15-F23]{ACGH}. If one takes a smooth plane quartic
$C$ and a Steiner set corresponding to some non trivial $\ga\in JC[2]$,
which is formed by six pairs of bitangents to $C$ in this case, then the six
intersection points of the pairs of bitangents sit on a conic in the same plane
as $C$, and the double cover of the conic ramified along these six points is the
genus 2 Prym curve corresponding to $C,\ga$, which in turn means that the
conic is a reincarnation of the dual Prym canonical system.
Note that in this case Property {\bf C} is
trivial as the quartic does not sit on any conic.

More recent results show that Properties {\bf A}, {\bf B}, {\bf C} with $n=1$ hold for
generic complex curves of genus $4$ (see \cite{LehGenus4}) and $5$ (see \cite{LehGenus5}).
In \cite{CelKulRenNamGenus4} the main result of \cite{LehGenus4} is extended in two
directions: to space sextics on singular quadrics and over more general fields.
In \cite{hanselman2024equationsgenus4curves} the spaces $V_{C,\ga}$ are interpreted fro genus $4$ in
terms of theta constants instead of theta hyperplanes.
\end{remark}
\begin{proposition} For $n=2$ the following equalities hold:
\[
\begin{aligned}
&\dim I_n(\gD(H^0(K_C+\ga)) =\binom{g-1}{2} + (g-1) \binom{g-2}{2} + 2\binom{g-1}{4},\\
&\dim \Sym^2H^0(2K_C)= \binom{3g-2}{2},\quad \dim \Sym^2\Sym^2H^0(K_C+\ga)= \binom{\binom{g}{2} + 1}{2}.
\end{aligned}\]
\end{proposition}
\begin{proof}
The first equality follows from \cite[Remark 5.1]{dolgachev_2003}, and the other
two are immediate.
\end{proof}
\begin{remark}\label{rem:codimension}
  Property {\bf C} fails for $n=2$ and $g < 6$ as the tacit surjectivity
  assumption of the map $\pi_{K_C+\ga}$ fails in these cases.
  Computing the dimensions above for genus $6,7,8,9$, the codimensions $d_\ga$
  from \Cref{codimV} become 50, 45, 21, -30 respectively.
  Thus, our codimension computation is vacuous already for genus 9. We also
  compute the increasing sequence of the dimensions of the spaces
  $\Sym^2\Sym^2H^0(K_C)$ for genus $6,7,8$:
\[\binom{\binom{6 + 1}{2} + 1}{2}=231,\quad \binom{\binom{7 + 1}{2} + 1}{2} =406,\quad  \binom{36 + 1}{2}= 666.\]
On the other hand, the numbers of pairs of theta hyperplanes arising from a Steiner set in genera $6,7,8$ are
$496, 2016, 8128$ respectively. Hence, in principle, Properties {\bf A}, {\bf B}, {\bf C} may all hold for genera $6,7,8$.

We now present the main result of the current study:
\end{remark}
\begin{theorem}\label{thm:main}
Properties $\text{{\bf A}}$, $\text{{\bf B}}$, $\text{{\bf C}}$ hold for $n=2, g=6,7$.
\end{theorem}
We now describe our approach to prove this theorem using ideas from \cite{LehGenus4,LehGenus5}, \cite{CelKulRenNamGenus4}. Indeed, we will make use of the following lemma. Later, in \Cref{r_cost}, we discuss key differences between the current work and these previous ones beyond the obvious technical difficulty of moving from $n=1$ to $n=2$ in property {\bf C}.
\begin{lemma}\label{L:C7L2}
Let $\cV/X$ be a vector bundle over a base $X$, and let $\cV_1,\dots \cV_n$ be sub-bundles of $\cV$.
Then the function $\dim\rmspan(\cV_1|_x,\ldots, \cV_n|_x)$ is lower semi-continuous on $X$,
and the function $\dim (\cap_{i=1}^n\cV_i|_x)$ is upper semi-continuous on $X$.
\end{lemma}
\begin{proof}
Classical -- see e.g. \cite[Corollary 7]{LehGenus4}.
\end{proof}
\subsection*{The ingredients of the proof of \Cref{thm:main}}
We start by arguing that \Cref{thm:main} follows from the statement of the theorem for a \emph{specific} curve thanks to \Cref{L:C7L2}. Indeed, each $q_{\{\gth,\gth+\ga\}}$ generates a rank 1 sub-bundle of $V_{2,\ga}$ over the family of genus $g$ curves $C$. Whence, it suffices to show that the dimension of the span of the image of $q_{\{\gth,\gth+\ga\}}$ is the same as the dimension of $V_{2,\ga}$ for a particular curve $C$, and similarly, it is enough to show that given some subset
$\{\ga_i\}\subset JC[2] \ssm\{0\}$ the intersection $\cap_iV_{2,\ga_i}C$ has codimension at least $\dim \Sym^2\Sym^2H^0(2K_C)$ for a particular curve.
To this end, we proceed as follows:
\begin{enumerate}
\item\label{i:outsideinterest} In \Cref{sec:certificationThetaHyperplanes}, we build a certification tool for a single theta hyperplane which relies on having a non-certified approximation of a theta hyperplane approximation and its
  tangency points.
    \item In \Cref{sec:certificationSteiner}, given a non-certified level 2 structure and the certified intersection points of the certified hyperplanes, we build a certification tool for subsets of some of the Steiner sets.
    \item In \Cref{sec:dimcert}, we build certification tools for the
      inequalities
      of $\dim V_{2,\ga}\geq d_\ga$ and $\codim\cap V_{2,\ga} \geq\codim\ker(\pi^2_{K_C})$, needed to
      prove Properties {\bf B}, {\bf C} respectively (we already proved the
      other direction of the first, and the other direction of the second
      follows assuming equality on the first).
 \item In \Cref{sec:program}, we introduce Wiman's sextic and the
   Fricke-Macbeath curve - our ``specific'' curves for $g=6$ and $7$
   respectively. Next we explain how to compute the non-certified theta
   hyperplanes and level 2 structure for them, and discuss the errors
   in these algorithms. We then continue to discuss errors incurred by
   performing numerical linear algebra computations and our own certification
   tools from the previous sections; concluding with the proof of
   \Cref{thm:main}. Finally, we also discuss in this section the testing of
   the certification code.
\end{enumerate}
As already mentioned in \Cref{sec:intro}, the certification tool in \Cref{i:outsideinterest} may be of independent interest in both computing and certifying intersection points acquiring multiplicities -- see \Cref{lma:F}, \Cref{rmr:grad_descent} \Cref{i:grad_descent} and \Cref{subsec:testing} \Cref{i:grad_decent_bound}.
\begin{corollary}\label{cor:saturation}
  For a generic curve of genus $6$ or $7$, the ideal $I_2(C)$ is the saturation of the ideal
  generated by $\ker(\pi^2_{K_C})$ inside the graded algebra $\oplus_{i=0}^\infty\Sym^{2i}H^0(K_C)$ at the irelevant
  ideal, and the canonical image of $C$ is the intersection of the nulls of the forms of $I_2(C)$.
\end{corollary}
\begin{proof}
By \Cref{thm:main}, we can recover the kernel of the map $\pi^2_{K_C}$.
Note that the short exact sequence
\[0\to I_2(C)\to\Sym^2H^0(K_C)\to H^0(2K_C)\to 0\]
gives after applying $\Sym^2$ the right exact sequence
\[I_2\otimes\Sym^2H^0(K_C)\to\Sym^2\Sym^2H^0(K_C)\overset{\pi_{K_C}^2}{\longrightarrow}\Sym^2H^0(2K_C)\to 0.\]
Thus the image of $\ker(\pi_{K_C}^2)$ under the multiplication map
\[\Sym^2\Sym^2H^0(K_C)\to\Sym^4H^0(K_C)\]
is the image of $I_2(C)\otimes\Sym^2H^0(K_C)$.
Our problem is to recover $I_2(C)$ from this data. To this end we consider the ideal
generated from this image in the graded ring $\oplus_{i=0}^\infty\Sym^{2i}H^0(K_C)$,
saturate it at the irrelevant ideal $\oplus_{i=1}^\infty\Sym^{2i}H^0(K_C)$, and take the degree two
part.

Finally recall that by the Enriques–Babbage theorem a non-hyperelliptic canonical curve $C$ is cut out by quadrics unless it is trigonal or a plane quintic. 
\end{proof}
Given a principally polarized Abelian variety, one can define the theta hyperplanes via
the Gauss
map. In explicit analytic terms, given a period matrix in the Siegel upper half space, which
parameterizes the Abelian variety, the theta hyperplanes are expressed in terms of the Riemann theta
function. The {\em Schottky problem} asks for a description of the (closure of the) locus of
Jacobians inside the (analytic) moduli of Abelian varieties (see \cite{Debarre85,Donagi88,Gru12} for
overview papers, and \cite{Shiota86,Shiota89,Krichever10} for proofs of Novikov and Welters
conjectures, which describe the Schottky
locus in terms of solutions to certain differential equations in theta functions).
Using the notation introduced at the beginning of Section~\ref{sec:main}, we
now can now transform the problem to a {\em completely algebraic, and so (almost - with the
  exception of 2) characteristic
free problem}
as well as describe the Schottky locus -- the locus of Jacobians in the moduli space of Abelian
varieties -- for genus $6$ or $7$.

We start by introducing two constructions
for principally polarized Abelian varieties (or ppavs),
analogous to those given earlier in the section for curves: Let $(A,\Theta)$ be a ppav; for each
\(\alpha \in A[2] \setminus \{0\}\), define
$\Sigma_\alpha := \left\{ \{\theta, \theta + \alpha\} \mid  \theta, \theta + \alpha \in A[2] \text{ are odd 2-torsion points} \right\}$.
Given a pair $\{\theta, \theta + \alpha\} \in \Sigma_\alpha$, define
$q_{\{\theta, \theta + \alpha\}} \in \left| \cO_{\PP T_0 A}(2) \right|$ to be the image of
$\{ [T_\theta \Theta], [T_{\theta + \alpha} \Theta] \} \in \Sym^2 \PP(T_0 A)^*$ under the natural map
$\Sym^2 \PP(T_0 A)^* \to |\cO_{\PP T_0 A}(2)|$, where  all
tangent spaces $T_p A$ are identified with $T_0 A$ via translation. Note that in the case where
$A = JC$ is the Jacobian of a curve, these definitions coincide with the earlier ones for $C$. We further define a subspace $I_A \subset \Sym^2(T_0 A)^*$ as the saturation of $\bigcap_{\alpha \in A[2] \setminus \{0\}} \mathrm{span} \left\{ \nu_2(q_{\theta, \theta + \alpha}) \right\}_{\{\theta, \theta + \alpha\} \in \Sigma_\alpha}\subset \Sym^2 \Sym^2(T_0 A)^*$ at the irrelevant ideal of the graded ring $\bigoplus_{i=0}^\infty \Sym^{2i}(T_0 A)^*$. 
\begin{corollary}\label{cor:Schottky}
Let $g$ be $6$ or $7$. The Schottky locus in genus $g$ is the closure of the moduli points
$[(A,\Theta)]$ such that the null set of $I_A$ is a smooth genus $g$ curve $C$, and
that $(JC, \Theta_C)\cong(A,\Theta)$.
\end{corollary}
\begin{proof}
Consider a ppav $(A,\Theta)$. If the null set of $I_A$ were to define a canonical curve, one might ask whether $(A,\Theta)$ arises as the Jacobian of that curve. However, outside the Schottky locus, there is no curve whose Jacobian is $(A,\Theta)$. Now suppose $(A, \Theta)$ lies in the Schottky locus, then there exists a curve $C$ such that $(A,\Theta)\cong (JC, \Theta_C)$. By \Cref{thm:main} and
\Cref{cor:saturation}, if $(A,\Theta)$ is generic inside this locus, then 
$I_A=I_2(C)$ and the null of $I_A$ is $C$.
\end{proof}
Note that constructing $(JC, \Theta_C)$ as well as verifying that a projective map defined by a
correspondence between matched theta hyperplanes extends to
isomorphism of ppavs is completely effective.
 
To the best of our knowledge, this is the only known effective and the only algebraic
characterization of generic points on the Schottky locus in genus \(g > 5\). For a
detailed overview of the effectiveness of various approaches to the Schottky problem, we refer the
reader to \cite{KS2013}.
\begin{remark}\label{r_cost}
The construction of curves in genus 4 \cite{LehGenus4} was completely algebraic, moreover, unlike \cite{LehGenus5,CelKulRenNamGenus4}, the proofs were not computer-aided. In~\cite{LehGenus5}, although the proof was computer-aided, a witness was provided that furnishes the proof of the main result, which could be verified by rank computations on non-mechanized scale. Moreover, as the theta hyperplanes were found algebraically, the proof and also the witness verification could have been done over any field, with the exception of some bad characteristics.

In the current study, we deviate from this pleasant scheme in two manners: First, we find the theta hyperplanes by approximate analytic methods via theta functions. In particular, we do not know over which number field they are defined. Second, the techniques we use to certify our numerical computations involve large-scale linear algebra computations. Nevertheless, we stress that \Cref{thm:main} could be stated over any field of characteristic other than 2.
\end{remark}
We close this section by outlining some directions to pursue in future.
\begin{question}
As pointed in \Cref{rem:codimension}, the dimension counts break down for genus strictly greater than 8. It would be pointless to consider $I_n(\gD (V))$ for $n>2$, as this ideal is generated in degree $2$ for any $n$. One might attempt instead to use syzygies for the ideal, namely relations among the $2$-forms.
\end{question}
\begin{question}
The requirement for the moduli points $[(JC, \Theta_C)],[(A,\Theta)]$ to be identical at the end of
\Cref{cor:Schottky} is aesthetically unappealing. Is it needed ?
\end{question}
\begin{question}
Is there a more conceptual argument than our brute-force computational approach ? e.g. considering
properties of symmetric powers of the Hodge and Prym vector bundles on the moduli space
$\mathcal{R}_g$ ? considering the behviour of our construction on nodal curves ?
\end{question}
%
\section{Numerically certifying a single theta hyperplane}\label{sec:certificationThetaHyperplanes}
%
In this section we assume that we are given a single putative approximation
$H'$ of a theta hyperplane to $C\subset|K_C|^*$,
together with approximations $P_i'$, of the points comprising $H'\cap C$
(which we denote by $P_i$). Our objective is to certify $H'$ and $P_i'$; thus
we want to bound the distance between $H'$ and the closest theta hyperplane
$H$, as well as the distances between the $P_i$'s and the tangency points of
$H$ and $C$. Of these two problems, the fact that we do not have exact
hyperplanes is the more challenging one to deal with.

To certify the theta hyperplanes we locally approximate the curve
about each point $P_i$ with a conic (under the assumption that the tangency points are not flexes, which we verify). We
refer to this construction as \emph{conic approximation}. The analysis of this
construction enables us to establish bounds on the accuracies of the hyperplanes;
thus ensuring their numerical stability. This plan presents three problems; the
first two of which originate from numerical inaccuracies in the input
hyperplanes $H'$:
First, we need to find a conic approximation and prove that it is adequate.
Next, we
need to prove that having a hyperplane ``close to being \emph{multitangent}" to
the conic approximation means that the ``candidate" of a theta hyperplane is
indeed close to the theta hyperplane. Our last problem is
adressing numerical issues arising from the differences between $P_i$ and $P'_i$.

The section is organized as follows: We construct the local conic
approximation of
the curve in \Cref{lma:PTR}, and bound its accuracy in \Cref{r_x}. We use
these bounds to prove \Cref{lma:F}: a certification tool bounding the
distance between $H,H'$, under the assumption that
we are given the points $P_i$ (and not $P_i'$). Finally, in \Cref{thm:pp} and
\Cref{thm:hh} we drop this assumption and bound the distances between $H,H'$ and $P_i,P_i'$.
\subsection*{Construction of the conic approximation to $C$ at a point $P$}
Let $Q_1,\ldots Q_{\dim I_2(C)}$ be a projective basis for $I_2(C)$. We will abuse
notations and denote the $g\times g$
symmetric matrices corresponding to the quadrics by $Q_i$ as well. Then $C$ is given by the projectivization of the null set of $P\mapsto \{P^t Q_iP\}_i$. In what follows, we work with projective coordinates and represent points as column vectors of size $g$.

\begin{lemma}[The conic approximation of an intersection of quadrics]\label{lma:PTR}
Let $P$ be a non-flex point on $C$, and let
$T, R$  be so that $P+xT + x^2R$ is a conic approximation of $C$ at $P$ then $T,R$ satisfy
\begin{equation}\label{eq:conicApproximationQuadrics}
0=2P^tQ_iT,\qquad0=2T^tQ_iT+P^tQ_iR,\qquad\text{ such that }T,R\neq 0.
\end{equation}
\end{lemma}
\begin{proof}
The point $P$ lies on each quadric $Q_i$. Hence, $T$ and $R$ are solutions of the following equations for any $i$:
$0=\frac{d}{dx} (P+xT)^tQ_i(P+xT)|_{x=0}$, and
$0=\frac{ d^2}{d x^2} (P+xT+x^2R)^tQ_i(P+xT+x^2R)|_{x=0}$.
These equations imply the equations in \eqref{eq:conicApproximationQuadrics}.
Finally, $R\neq 0$ since $P$ is not a flex point.
\end{proof}
\begin{definition}\label{dfn:M_TandM_R}
We represent points $\PP^{g-1}$ -- specifically $P$ -- as norm $1$ vectors in $\CC^g$.
We take $T$ to be orthogonal to $P$ with $|T|=1$, and consider $R$ to be the smallest norm
solution of the linear system arising from \eqref{eq:conicApproximationQuadrics}. These last requirements may be formulated as follows: First, denote by $M$ the $\dim I_2(C)\times g$ matrix whose rows are ${\{P^tQ_i\}}_i$. We note that the co-rank of $M$ is 2. Indeed, if there were other vectors in the kernel than $P,T$ then $P$ would not be a smooth point of $C$. To solve $T$, we look at the kernel of the $(\dim I_2(C) + 1)\times g$ matrix $M_T:=\left(\begin{smallmatrix}M\\P^\dagger \end{smallmatrix}\right)$, where $\dagger$ denotes the complex
  conjugate transpose operator. To solve $R$,  denote $M_{R}:=\left(\begin{smallmatrix}M_T\\T^\dagger \end{smallmatrix}\right)$, and solve the equation
    \begin{equation}\label{e:rqr}
    \quad M_R R=-\frac{1}{2}\left(\begin{smallmatrix}
        T^tQ_1T\\ \vdots\\ T^tQ_{\dim I_2(C)}T \\ 0 \\ 0
    \end{smallmatrix}\right).
    \end{equation}
\end{definition}

We now introduce a couple of approximation results for points close to our curve $C$. The need for these results is twofold. The first is that we are about to express the curve locally as a ``conic plus something small". The second is that we will have to bound the effect of computational errors to certify our computations, as in practice, these are our input points. To this end, we will assume that $\sqrt{\sum_i|{\tilde{P}}^tQ_i\tilde{P}|^2} < \gep$ for some ${\tilde{P}}\in \CC^g$. For our purposes, this ${\tilde{P}}$ will either be an approximation of some point $P$ by means of conic approximation, or a computed version of a point $P\in C$. We let $\tilde{M}$ be the matrix whose rows are ${\tilde{P}}^tQ_i$. As usual, we denote the $j$th singular value of a matrix by $\sigma_j$.

Recall that the curve $C$ is a local complete intersection. This statement can be made effective as follows: The row vectors $P^tQ_i$s span the gradients of the quadrics in $I_2(C)$ at a point $P$. Hence, if we consider the matrix built of these row vectors then its top $g-2$ singular values are the only non-trivial ones. Also, the corresponding singular vectors of this matrix are the gradients of quadrics in $I_2(C)$ which span the normal space of the curve at $P$. We use this observation to set up the proof of \Cref{lma:close_to_P} as follows:
\begin{lemma}\label{lma:close_to_P}
Let $\tilde{Q}_1,\ldots, \tilde{Q}_{g-2}$ be the quadrics corresponding to
the top $g-2$ singular vectors of $\tilde{M}$, then there is a point $P\in C$ such that $|\tilde{P}-P| <r$ where $r\in \RR$ is such that $r<(|({\tilde{P}}^t\tilde{Q}_i\tilde{P})_i|- gr^2)/2\sigma_{g-2}(\tilde{M})$.
\end{lemma}
\begin{remark}[A note on distances] We consider points in $\PP^{g-1}$, which are represented by points in $\CC^g\smallsetminus\{0\}$, considered as a $\CC^*\cong\RR^+\times\mathbb{S}^1$ bundle
over $\PP^{g-1}$, where the $\RR^+$ coordinate of the bundle is the $L_2$ norm.
Denoting $x:=|(P_1/|P_1|)\cdot (P_2/|P_2|)|$, recall that
the Fubini-Study is defined as $\arccos(x)$; we use instead the distance
$|P_1-P_2|=\sqrt{|P_1|^2+|P_2|^2-2|P_1\cdot P_2|}$, and for representatives whose
norms are not necessarily 1 (but rather boundedly close to 1). By the Taylor
series for $\cos$, the function
$\frac{\arccos(x)}{\sqrt{1-x}}$
is monotonically decreasing from $[\pi/2]$ to $1$ on $[0, 1]$, hence
if one insists on controlling the quotient of the two metrics, it suffices to
control $|P_1|, |P_2|$.
\end{remark}
\begin{proof}[Proof of \Cref{lma:close_to_P}]
Consider the map
\[\begin{aligned}
    B(0,r)^{g-2}  &\to \CC^{g-2}\\
    z&\mapsto (\tilde{Q}_1(z),\ldots,\tilde{Q}_{g-2}(z))=((\tilde{P}+z)^t\tilde{Q}_i(\tilde{P}+z))_i=({\tilde{P}}^t\tilde{Q}_i\tilde{P})_i+2\tilde{M}z+(z^t\tilde{Q}_iz)_i.
\end{aligned}\]
Thus, we wish to solve the equation $0=({\tilde{P}}^t\tilde{Q}_i\tilde{P})_i+2\tilde{M}z+(z^t\tilde{Q}_iz)_i$. Our map is open, and
each of the coefficients in the degree $2$ tail is bounded by $r^2|\tilde{Q}_i|$. To complete the proof we argue that
$\sqrt{\sum_{i=1}^{g-2}|\Tilde{Q}_i|^2}<\sqrt{(g-2)\cdot g}<g.$
\end{proof}
Note that in the proof above we merely bound the norm, and thus we do not have to move to real coordinates, which we would have had to do if we wanted to explicitly refer to the closest point. We also remark that we make use of this technique of using effective forms of the implicit and inverse function theorem later in this section; e.g. in \Cref{lma:F}.
\begin{definition}\label{dfn:P_x}
  Considering complex parameter $x$, we denote
  $P_x:=P + x T + x^2 R,\quad T_x:=T+2xR$, let $M_x$ to be the matrix whose rows
  are $P_x^{t}Q_i$, let $M_{T,x}:=\left(\begin{smallmatrix}M_x\\P_x^\dagger\end{smallmatrix}\right)$, and
    $M_{R,x}:=\left(\begin{smallmatrix}M_{T,x}\\T_x^\dagger\end{smallmatrix}\right)$. Let $v_x$ be the closest point on $C$ to $P_x$. Finally, we denote the vector
      whose entries are $T^tQ_iR$ by $W_T$, and the vector whose entries are $R^tQ_iR$ by $W_R$. We refer to \Cref{fig:quartic} for intuitive views of $P, P_x, v_x$ and the tangency point. Finally we set $q_m:=\sqrt{\sum_{i=1}^{\dim I_2(C)}|Q_i|^2}$, where the norms are operators norms.
\end{definition}
\begin{figure}
    \centering
\includegraphics[scale=0.2]{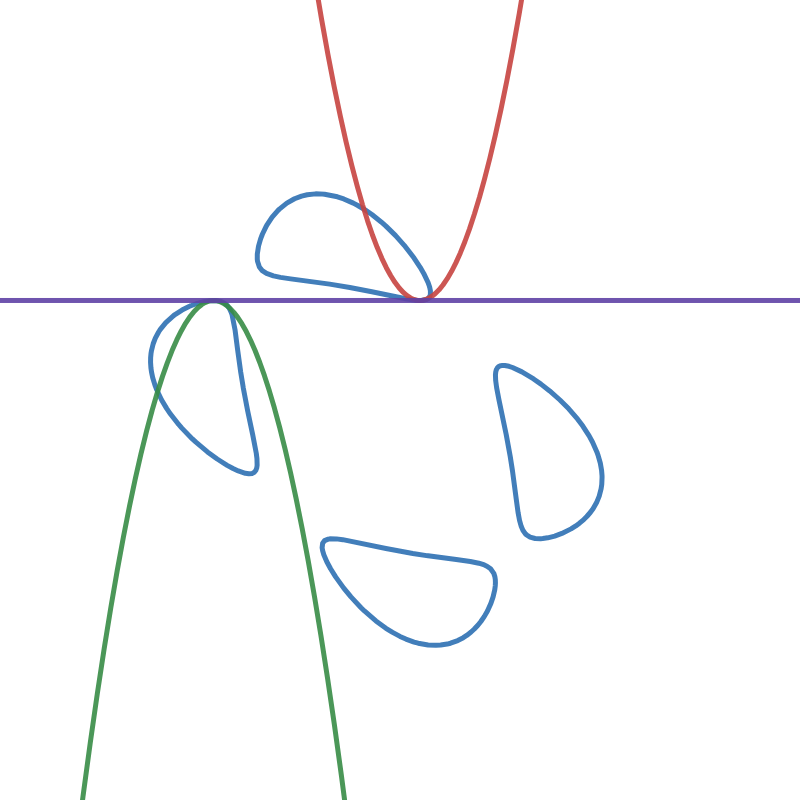}
\includegraphics[scale=0.2]{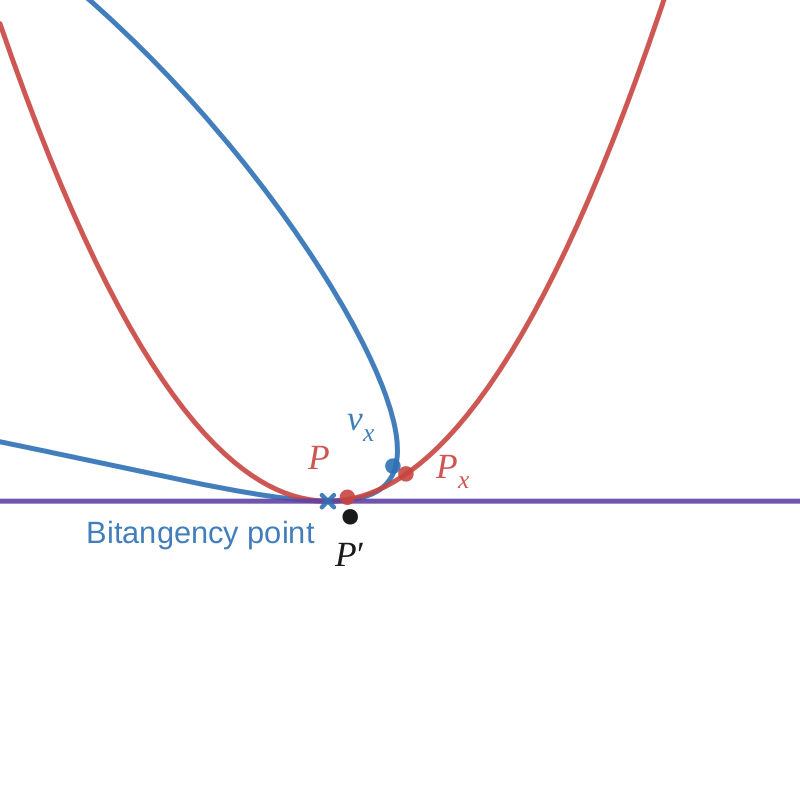}
    \caption{An illustration of the conic approximation argument for a plane quartic. Note that both $P,v_x$ and the tangency point are on the quartic curve, and that the conic approximation is at $P$ -- not at the tangency point.}
    \label{fig:quartic}
\end{figure}
The first of the two approximation tasks mentioned before \Cref{lma:close_to_P} is to show that the difference between the curve and its conic approximation is small. In essence, it means that we want to prove the existence of $r,\kappa_0,\kappa_1$ so that for $|x|<r$, we have $|P_x-v_x|<\kappa_0|x|^3,\frac{d}{dx}|P_x-v_x|<\kappa_1|x|^2$. Sadly, the actual expressions one attains for these
constants are complicated.
\begin{theorem}[Effective bounds on the quality of the conic approximation]\label{r_x}
  Assume that $|x|$ is smaller than
{\small
\[
r_x:=\min\Big(
\frac{1}{4|R|},
\frac{\sigma_{g-2}(M)}{2q_m},
\frac{|W_T|}{2|W_R|},
\left(\frac{\sigma_{g-2}(M)^2}{g|W_T|}\right)^{1/3},
\ \frac{\sigma_{g-1}(M_T)}{8q_m},
\ \left(\frac{\sigma_{g-1}(M_T)\sigma_{g-2}(M)}{4(q_m+2)|W_T|}\right)^{1/3},
\ \frac{\sigma_{g-2}(M)}{q_m+1}
\Big)
\]
}
then
\[
|P_x-v_x| / |x|^3 < \kappa_0:=|W_T| /\sigma_{g-2}(M),\qquad
\left|\frac{d(v_x-P_x)}{dx}\right|/|x|^2 < \kappa_1:=8|W_T|/\sigma_{g-1}(M_T).
\]
\end{theorem}
\begin{proof}
We reformulate \Cref{dfn:M_TandM_R} for the point $v_x$ in the role of $P$:
Let $M_{v_x}$ be the matrix whose rows are $v_x^tQ_i$, and
$M_{T,v_x}:=\left(\begin{smallmatrix}M_{v_x}\\ v_x^\dagger\end{smallmatrix}\right)$.
Denote by $T_{v_x}$ a normalized vector that spans the kernel of $M_{v_x}$,
i.e., $T_{v,x}$ spans the tangent space to $C$ at $v_x$.
The proof is technical, and involves straightforward but tedious Frobenius norms
estimations of the matrices involved in the linear systems defining
$T_x, T_{v_x}$.

{\em We start by bounding $|P_x- v_x|$.}
By the definition of $P,T, R$ we have:
  \[
  P_x^tQ_iP_x=(P+xT+x^2R)^tQ_i(P+xT+x^2R)=2x^3T^tQ_iR+x^4R^tQ_iR=
  2x^3(W_T)_i+x^4(W_R)_i.
  \]
  Per \Cref{lma:close_to_P}, to achieve
  $|P_x-v_x|<|x^3W_T| / 2\sigma_{g-2}(M_x)$ is suffice to have:
  \[
  |x^3W_T/2\sigma_{g-2}(M_x)| < \left(|2x^3W_T|-|x^4W_R| - |g(x^3W_T/2\sigma_{g-2}(M_x))^2)|\right)/2\sigma_{g-2}(M_x).
  \]
  Subtracting the LHS this is equivalent to
  \[
  0 <\frac{|x|^3}{2\sigma_{g-2}(M_x)}\left(|W_T|-|xW_R| - g|x|^3(|W_T|/2\sigma_{g-2}(M_x))^2\right),
  \]
  and indeed,
  {\small
  \[
  |W_T|-|x|\cdot|W_R| - g|x|^3\left(\frac{|W_T|}{2\sigma_{g-2}(M_x)}\right)^2>
  |W_T|-\frac{|W_T|}{|2W_R|}|W_R|-g\frac{\sigma_{g-2}(M)^2}{g|W_T|}\left(\frac{|W_T|}{2\sigma_{g-2}(M_x)}\right)^2 =
   |W_T|\left(1-\frac{1}{2}-\frac{1}{4}\right).
\]}
  Since $|T|=1$, we have from the definition of the Frobenius norm that:
  \begin{equation}\label{eqn:frob_M}
    ||M_x-M||_F^2<q_m^2(|T|\cdot|x|+|R|\cdot|x|^2)^2<q_m^2(|x|+|R|\frac{1}{4|R|}|x|)^2=((5/4)q_m|x|)^2,
  \end{equation}
  and since $|x|<\sigma_{g-2}(M)/2q_m$ we have by \Cref{thm:sing_sum}:
  \[
  |\sigma_{g-2}(M_x)-\sigma_{g-2}(M)|<(5/4)q_m|x| < (5/8)\sigma_{g-2}(M).
  \]
  Hence, $|P_x-v_x|<\gep_{P,x}:=|x^3W_T| /\sigma_{g-2}(M)$; thus establishing the
  claim about $\kappa_0$.

  {\em We now bound $|T_x-T_{v_x}C|$.}
  Per the definition of $P, T, R$, we have
  \[
  (M_{T,x}\cdot T_x)_i=(T+2xR)^tQ_i(P+xT+x^2R) =
  2x^2(T+xR)^tQ_iR=2x^2((W_T)_i+x(W_R)_i).
  \]
  Hence, the distance from $T_x$ to $\ker M_{T,x}$ is bounded by
  $|2x^2(W_T+xW_R)|/\sigma_{g-1}(M_{T,x}$).
  By \Cref{prp:kercont} we get:
  \[
  |T_{v_x}C-T_x| < |M_{T,v_x}-M_{T,x}|/(\sigma_1(M_{T_,x})-|M_{T,v_x}-M_{T,x}|)+2|x|^2\cdot|W_T+xW_R|/\sigma_{g-1}(M_{T,x}).
  \]
  By the definition of Frobenius norm,
  \Cref{eqn:frob_M} and the definition of $\gep_{P,x}$ above, we see that
  \[||M_{T,x}-M_T||_F^2<((5/4)q_m\cdot|x|)^2+\gep_{P,x}^2,\]
  whereas simply by the definition of the Frobenius norm, and the
  definitions of $M_{T,x}, M_{T,v_x}$ we have:
  \[|M_{T,v_x}-M_{T,x}|<||M_{T,v_x}-M_{T,x}||_F< \sqrt{1+q_m^2}\gep_{P,x}.\]
  Using these bounds we may bound $|T_{v_x}C-T_x|$ with
  {\small
  \[\begin{aligned}
  \ &\frac{\sqrt{1+q_m^2}\gep_{P,x}}{\sigma_1(M_T)-\sqrt{(5/4)^2q_m^2|x|^2+\gep_{P,x}^2}-\sqrt{1+q_m^2}\gep_{P,x}}+\frac{2|x|^2(|W_T|+|x|\cdot|W_R|)}{\sigma_{g-1}(M_T)-\sqrt{(5/4)^2q_m^2|x|^2+\gep_{P,x}^2}}\\
  <&\frac{(1+q_m)\gep_{P,x}}{\sigma_1(M_T)-2q_m|x|-2\gep_{P,x}-q_m\gep_{P,x}}+\frac{2|x|^2|W_T|(1+|W_R|/2|W_R|)}{\sigma_{g-1}(M_T)-2q_m|x|-\gep_{P,x}}
  <\frac{(1+q_m)\gep_{P,x}+3|x|^2\cdot|W_T|}{\sigma_{g-1}(M_T)-2q_m|x|-(q_m+2)\gep_{P,x}}\\
  <&\frac{(1+q_m)\frac{\sigma_{g-2}(M)}{q_m+1}|x|^2\frac{|W_T|}{\sigma_{g-2}(M)}+3|x|^2\cdot|W_T|}{
    \sigma_{g-1}(M_T)-2q_m\frac{\sigma_{g-1}(M_T)}{8q_m}-(q_m+2)\frac{\sigma_{g-1}(M_T)\sigma_{g-2}(M)}{4(q_m+2)|W_T|}\cdot\frac{|W_T|}{\sigma_{g-2}(M)}}=
  \frac{(1+3)|x|^2|W_T|}{\sigma_{g-1}(M_T)(1-\frac{1}{4}-\frac{1}{4})}=
  8|x|^2\frac{|W_T|}{\sigma_{g-1}(M_T)}.
  \end{aligned}\]
  }
  {\em Finally we bound $|\frac{d}{dx}(P_x-v_x)|$.}
  Let $q_1:\CC\to C$ be a local coordinate above for $C$ about $v_x$ (i.e.
  $q_1(0)=v_x$, and let $q_2$ be the map $x\mapsto P_x$.
  We consider the differential of the map $x,y\mapsto q_1(x)-q_2(y)$ at $x,0$.
  This derivative is of the form $(T_x, -aT_{v_x}C)$ for some constant $a$,
  which depends on the choice of $q_2$, and is controlable by
  composing $q_2$ with multiplication by a scalar. Hence,
  \[
  \frac{d}{dx}(P_x-v_x)|=\min_a|T_x-a T_{v_x}C| < 8|x|^2\frac{|W_T|}{\sigma_{g-1}(M_T)}.
  \]
\end{proof}
\subsection*{From an approximate theta hyperplane to a hyperplane}
In what follows, we apply the conic approximation for $g-1$ points
simultaneously. To this end, we define
$P_i, T_i, R_i\in\CC^g,\ P_{i,x},v_{i,x}$ as
in \Cref{dfn:M_TandM_R}, \Cref{dfn:P_x} and $r_{i,x}$ as in \Cref{r_x}.
We also define $c_i:=|R_i|,\tilde{R}_i:=R_i/c_i$ and $u_{i}(x):=v_{i,x}-P_{i,x}$.
Then by \Cref{r_x}, the function $u_i(x)$ satisfies $e_i:=|u_i(x)| < \kappa_{0,i} |x^3|$,
$e'_i:=|d u_i(x)/d x| < \kappa_{1,i}|x|^2$ for $\kappa_{0,i},\kappa_{1,i}$ as in \Cref{r_x}.
We also define:
\begin{align}\label{eq:definitionF}
F:\prod^{g-1}_{i=1} \BB_i&\to\CC^{g-1}\\
(x_1,\ldots,x_{g-1})&\mapsto \left(\nabla_{x_j} v_{j,x_j}\wedge \bigwedge_{i=1}^{g-1} v_{i,x_i}\right)_j=\left(\det(\nabla_{x_j} v_{j,x_j}, v_{1,x_1},\ldots, v_{g-1,x_{g-1}})\right)_j. \nonumber
\end{align}
The wedge product $\wedge_i v_{i,x_i}$ is an equation satisfied by the hyperplane
passing through small perturbations on $C$ of the points $P_i$ whereas
the $\nabla_{x_j} v_{j,x_j}$ spans the tangent
space to the curve at the perturbation of $P_j$. Hence, the $j$th coordinate of
$F$ essentially checks by how much the tangent direction
deviates from the hyperplane determined the wedge product.
Whence, our plan of attack is to formalize and prove the following claim:
if $F(0)$ is close to $0$,
and the Jacobian of $F$ is sufficiently far from
being singular, and $F$ is well approximated by 1st order approximation, then
$F$ attains $0$ near the origin.

In order to give the needed quantitative bounds for the Jacobian matrix of $F$, denoted $\cJ_F$, we set the following notations:
\[\begin{aligned}
\mu_i:\RR^+&\to\RR^+\qquad&\eta_i:(\RR^+)^{n-1}&\to\RR^+\\
r&\mapsto 1+2|c_i|r+\kappa_{1,i}r^2,\qquad&(r_1,\ldots \hat{r_i}\ldots,r_n)&\mapsto\prod_{j\neq i}(1+r_j+|c_j|r_j^2+\kappa_{0,j}r_j^3),
\end{aligned}\]
where the hat symbol \, $\hat{}$ \, removes the entry in the corresponding position. We will customize the notation $\eta_i(r_1,\ldots \hat{r_i}\ldots,r_n)$ by using $\eta_i(r_i)$ instead.
For a function $\psi:\RR^N\to\RR^M$ or $\psi:\CC^N\to\CC^M$, we have the following notation for the expansion of $\psi$ around $0$:
\begin{equation}\label{eq:jacobian}
\Term_1(\psi)(x):=\cJ_\psi\restriction_0 x,\quad \Tail_2(\psi):=\psi-\psi(0)-\Term_1(\psi).
\end{equation}
\begin{lemma}\label{lma:F}
In the setting above:
  \[{\cJ_F}_{i,j}=\begin{cases}2c_i\det(\tilde{R}_i,P_1,\ldots, P_n),&i=j\\
\det(T_i, P_1,\ldots , P_{j-1},T_j,P_{j+1},\ldots , P_n),&i\neq j\end{cases}.\]
Moreover, if $|H'|=1$ and
\begin{equation}\label{eq:lma_F}
\left|\langle\wedge_i v_{i,x}(x_i)\restriction_0, H' \rangle\cdot\cJ_F\restriction_0^{-1}\left((\nabla_{x_j}{v_{j,x}}\restriction_0) \wedge H'\right)_j\right| < \min_ir_i-\Tail_2(\gS_i {\mu_i(r_i)\eta_i(r_i)})/\sigma_{n}(\cJ_F\restriction_0)
\end{equation}
where the norm is the symmetric tensor norm,
then there is a solution for $F=0$ on $\prod_i \BB_i$.
\end{lemma}
\begin{proof}
The $k$th entry of the image of $F$ is
\[
\det(T_k+2x_kc_k\tilde{R}_k+du_k/dx_k, P_1+x_1T_1+c_1x_1^2\tilde{R}_1+u_1,\ldots, P_n+x_nT_n+c_nx_n^2\tilde{R}_n+u_n).
\]
Expanding the expression above by their $x_j, u_j, du_j/dx_j$ terms , we get $T_k\cdot\wedge_i P_i$ plus
\[
2c_kx_k\det(\tilde{R}_k,P_1,\ldots, P_n) + \sum_{j\neq k} x_j\det(T_k, P_1,\ldots P_{j-1}, T_j,P_{j+1}\ldots P_n)
\]
plus $\Tail_2F,$ which is of the form:
{\small\[
\sum (\text{polynomial in } u_j, du_j/dx_j, x_j,c_j)\times
\text{a determinant whose rows are among }P_i,\tilde{R}_i,T_i.
\]}
To bound $|\Tail_2F|$ we first note that the norms of the rows
of the determinants are all $1$, and so the determinants are bounded by $1$.
Next we observe that
if we substitute each occurrence of $u_j,du_j/dx_j,x_j$ by
$\kappa_{0,j}r_j^3, \kappa_{1,j}r_j^2, r_j$ respectively, and each of the determinants by $1$,
we get exactly  $\Tail_2\left(\gS_i \mu_i(r_i) \eta_i(r_i)\right)$.
As the spectral norm of an
operator is bounded by the norm of the sum
of the entries, we get
$|\Tail_2(F)\restriction_{\prod \BB_i}|\leq \Tail_2\left(\gS_i \mu_i(r_i) \eta_i(r_i)\right)$.
Expanding $F$ around $x=0$ we get $F = F(0) + (\cJ_F\restriction_0)\cdot x + \Tail_2(F)$. Hence, the image of $F$ contains an open ball of radius $\sigma_{n}(\cJ_F\restriction_0)\min_ir_i-\sup_{\prod\BB_i}\Tail_2F$ around $F(0)$.
As $H'$ is the normalization of $\wedge_iP_i=\wedge_iv_{i,0}$, up to a
norm 1 scalar, the last part of the result follows.
\end{proof}
\subsection*{Dealing with inaccuracies of the input points}
We finally come to our last barrier in this section: improving \Cref{lma:F}
to deal with the approximations of the points $P_i$ given by $P'_i$ instead
of the $P_i$ themselves.
\begin{definition}[Notations for computing the numerical stability]\label{dfn:stability_notations}
Recall that $H'$ denotes an approximated theta hyperplane. We normalize $H'$ so that $|H'|=1$. Let $P,P'$ be representation in $\CC^g$ of points in $\PP^{g-1}$ such that
\[
\PP P\in C\cap Z(H'),\qquad|P|=|P'|=1,\qquad|P-P'|<\gep_0,\qquad \sum_i|{P'}^tQ_iP'|^2<\gep_0^2.
\]
Denote by $T',R',M_{T'},M_{R'}$ the machine computed analogs of
$T,R,M_T,M_R$ from \Cref{lma:PTR}, represented with floating-point numbers.
More precisely, $T'$ is the computed lowest singular vector of $M_T'$ of norm 1 whereas  $R'$ is computed via \Cref{e:rqr}, but reformulated in terms of the computed $P'$ and $T'$. The requirement on $|P-P'|$ follows from the requirement on $\sum_i|{P'}^tQ_iP'|^2$ for an appropriate choice of $\gep_0$ by \Cref{lma:close_to_P}. We refer again to \Cref{fig:quartic} for the relation between $P'$ and $P$.
\end{definition}
\begin{proposition}[Bounding the difference $|T-T'|$]\label{prp:tcont}
Let $T'$ and $T$ be the vectors defined in \Cref{dfn:stability_notations}. Then
\begin{equation}\label{eq:eps1}
    |T'-T|<\gep_1:=1 + |T'| - 2\cos(\verb|E|\cdot\sigma_1'(M_{T'})/{\verb|ddisna|}(M_{T'}, g))+\sigma_g(M_T)^+/\sigma_{g-1}(M_T)^-,
\end{equation}
where $\sigma_i(M_T)^\pm$ (see \Cref{dfn:abovebelow}) may be bounded via
\[
|\sigma_i'(M_{T'})-\sigma_i(M_T)| < \sigma_1'(M_{T'})\cdot\verb|E|+\gep_0\sqrt{1+\sum_j(|Q_j|+g\verb|E|)}.
\]
\end{proposition}
\begin{proof}
We start with the bound on $\sigma_i(M_T)^\pm$:
Per \Cref{prp:computed_sing}, \Cref{thm:sing_sum} we have
\[
|\sigma'_i(M_{T'})-\sigma_i(M_T)|<\sigma'_1(M_{T'})\cdot {\verb|E|} + |\sigma_i(M_{T'})-\sigma_i(M_T)|<\sigma'_1(M_{T'})\cdot {\verb|E|} +|M_{T'}-M_T|.
\]
We bound the latter with $||M_{T'}-M_T||_{F}$: The difference
between the last line of $M_{T'}$ and the last row of $M_T$ is simply $({P'} - P)^\dagger$, and the difference between all the other rows is given by $Q_i(P'-P)$. The quantity $g$\verb|E| accounts for the $g$ operations involved in computing each $Q_iP'$ as the entries of each $Q_i$ are small integers and the norm of the entries of $P'$ are $<1$. Hence, $(||M_{T'}-M_T||_F/\gep_0)^2<1+\sum_j(|Q_j|+g\verb|E|)$.

For the first three summands in \Cref{eq:eps1}, per {\verb|LAPACK|}'s documentation of the subroutine {\verb|ddsina|} as discussed in \Cref{rmr:lapack_oracle}, the angle between the true $i$th singular vector and the computed singular vector of a matrix $A$ is given by $\verb|E|\cdot\sigma_1'(A)/{\verb|ddisna|}(A, i)$. The last summand follows from \Cref{prp:kercont}.
\end{proof}
\begin{proposition}[Bounding the difference $|R-R'|$]\label{prp:rcont}
  \ \\Let $\gep_{M_{R'}}:=(1+2{\verb|E|})\sqrt{\gep_0^2(1+(\sum_j|Q_j|+g\verb|E|)) + |T-T'|^2}$, let
 {\small\[
     \gd_R:=(\text{the sum of the last $\dim I_2(C) + 2 - g$ }\verb+zgelss+\text{ return values})\cdot (1+(\dim I_2(C) + 2-g)\verb|E|),
     \]}
 where the \verb|zgelss| return values are per \verb|LAPACK|s accuracy estimate
 of the system
$$M_{R'} R'=-\frac{1}{2}\left(\begin{smallmatrix}{T'}^tQ_1T'\\ \vdots\\ {T'}^tQ_{\dim I_2(C)}T' \\ 0 \\ 0 \end{smallmatrix}\right)$$ per \Cref{rmr:lapack_oracle},
 and let $\xi_T$ an upper
  bound on the top singular value of the matrix whose $j$th row is given by
  ${T'}^tQ_j$. Then
  $|R-R'|$ is bounded by
  {\small
\begin{equation}\label{eq:eps_2}
  \gep_2 := \frac{1}{2}|({T'}^tQ_iT')_i|\cdot \left(\frac{\gep_{M_R}}{\sigma_g{M_R}^-(\sigma_g(M_R)^--\gep_{M_R})}+(\frac{1}{2}(2\xi_T\gep_1 +\gep_1^2q_m)+\gd_R)/(\sigma_g(M_R)^--\gep_{M_R})\right),
\end{equation}}
where $\sigma_i(M_R)^\pm$ may be bounded via
\[
|\sigma_i'(M_{R'})-\sigma_i(M_R)| < \sigma_1'(M_{R'})\cdot\verb|E|+\gep_{M_{R'}}.
\]
\end{proposition}
\begin{proof}
  We start by bounding $\sigma_i(M_R)^\pm$:
  Note that $||M_{R'}-M_R||_F <\gep_{M_{R'}}$ by the definition of the matrix $M_R$, which has been formed from the matrix $M$ with the additional two equations as introduced in \Cref{dfn:M_TandM_R}. As in the proof of \Cref{prp:tcont},
  we deduce that  $|\sigma'_i(M_{R'})-\sigma_i(M_R)|< \sigma'_1(M_{R'})\cdot\verb|E| + \gep_{M_{R'}}$.
 Per the definition of $\xi_T$ we have:
$|({T'}^tQ_iT')_i-(T^tQ_iT)_i|< \frac{1}{2}(2\xi_T\gep_1 +\gep_1^2q_m)$,
hence the result follows from \Cref{prp:solcont}.
\end{proof}
In what follows, we work out the version of \Cref{lma:F} which accounts for the numerical errors $|P_i-P'_i|$, and the resulting errors $|T_i-T'_i|, |R_i-R'_i|$ which we bounded above. We do not elaborate here on the accuracies of the terms of $r_x$ or $\kappa_{0,i}$ and $\kappa_{1,i}$, as for our purposes it suffices to
show that all these numbers are $O(1)$, whereas the numerical errors in
$P_i, T_i, R_i$ are orders of
magnitude smaller (we discuss this in detail in \Cref{sec:program}).
This is not the case for \Cref{lma:F} as the expressions appearing there, e.g.
the top singular value of $\cJ_F$, are complicated, and so the accumulated error
analysis is complicated.

Suppose we are given $g-1$ many points $P'_{i}$ that are estimates of the true tangency points. Let $P_{i}, P'_{i}$ be representations in $\CC^g$ of points in $\PP^{g-1}$ such that
\[
\PP P_{i}\in C\cap Z(H'),\qquad|P_{i}|=|P'_{i}|=1,\qquad|P_{i}-P'_{i}|<\gep_{0,i},\qquad \sum_j|{P'_{i}}^tQ_jP'_{i}|^2<\gep_{0,i}^2.
\]
Let $T_i, T'_i, \tilde{R}_i, \tilde{R}'_i$ etc. be all the relevant objects mentioned in \Cref{dfn:stability_notations} with the extra subscript $i$ which refers to the $i$th putative tangency points $P_{i}$ and $P'_{i}$. So e.g. $\epsilon_{1,i}$ is the bound for $|T_i'-T_i|$ considering the corresponding $T_i$ and $T_i'$ and similarly, $\epsilon_{2,i}$ concerns the bound on $|R_i-R'_i|$.

We now rewrite the function defined in \Cref{eq:definitionF} in its computed version, using the prime symbol for denoting a computed version as usual.
\begin{align}\label{eq:definitionF'}
F':\prod_i \BB_i&\to\CC^{g-1} \\
(x_1,\ldots,x_{g-1})&\mapsto \left(\nabla_{x_j} f'_j(x_j)\wedge \bigwedge_{i=1}^{g-1} f'_i(x_i)\right)_j=(\det(\nabla_{x_j} f'_j(x_j), f'_1(x_1),\ldots, f'_{g-1}(x_{g-1})))_j, \nonumber
\end{align}
where $f'_i:\BB_i\to\CC^n$ is given by $f'_i(x_i)=P'_i+x_iT'_i+c'_ix^2\tilde{R_i'}$.
\begin{lemma}\label{lma:F_with_errors} The Frobenius norm of the difference
${\cJ_F}\restriction_0-{\cJ_{F'}}\restriction_0$ is bounded by
\[\begin{aligned}
\gep_{\cJ_F}:=&(1+\verb|E|)^{g^3}\Big(2\sum_i\gep_{2, i}(1+\gep_{2,i})\prod_{k\neq i}(1+\gep_{0,k})\\
+&2\sum_i(|R'_i|+\gep_{2,i})((1+\gep_{2,i})\prod_{k\neq i}(1+\gep_{0,k})-1)
+\sum_{i,j}\big((1+\gep_{1,i})(1+\gep_{1,j})\prod_{k\neq i,j}(1+\gep_{0,k})-1\big)\Big).
\end{aligned}\]
\end{lemma}
\begin{proof}
The norm estimation is based on the computations in the proof of \Cref{lma:F}. We denote
\[
\gd_{P_i}:=P'_i-P_i, \gd_{T_i}:={T'}_i-T_i, \gd_{R_i}:={\tilde{R}}'_i-\tilde{R}_i,\gd_{c_i}:=c'_i-c_i.
\]
We also set $X_I$ to be the matrix whose first row is $\gd_{R_i}$ if $i\in I$, and $R_i$ otherwise, and whose $k>1$ row is $\gd_{P_k}$ if $k\in I$ and $P_i$ otherwise. The differences between the $g$ diagonal elements in \Cref{lma:F} can be written as
\[\begin{aligned}
&2(c_i + \gd_{c_i})\det(\tilde{R}_i+\gd_{R_i},P_1+\gd_{P_1},\ldots, P_n+\gd_{P_n})-2c_i\det(\tilde{R}_i,P_1,\ldots, P_n)\\
&=2\gd_{c_i}\det(\tilde{R}_i+\gd_{R_i},P_1+\gd_{P_1},\ldots, P_n+\gd_{P_n})+2c_i\sum_{\substack{I\subset \{1,\dots, n+1\} \\ \#I > 1}}\det(X_I).
\end{aligned}\]
The norms of these summands are bounded by the following quantities respectively:
\[
2|\gd_{c_i}|(1+|\gd_{c_i}|)\prod_{i\neq k}(1+\gep_{0,k}),\qquad2|c_i|((1+|\gd_{c_i}|)\prod_{i\neq k}(1+\gep_{0,k})-1).
\]
Likewise, we now compute the differences between the non-diagonal elements in \Cref{lma:F}, of which there are $g(g-1)$, are given by
\[
\det(T'_i, P'_1,\ldots , P'_{j-1},T'_j,P'_{j+1},\ldots , P'_n)-\det(T_i, P_1,\ldots , P_{j-1},T_j,P_{j+1},\ldots , P_n).
\]
A  similar argument for the diagonal elements gives us the bound
\[
(1+\gep_{1,i})(1+\gep_{1,j})\prod_{k\neq i,j}(1+\gep_{0,k}) - 1.
\]
To conclude the proof we observe that $|\gd_{c_i}|<\gep_{2,i}$, and that $|c_i|=|R_i|$, then we count the number of arithmetic operations involved in the computation.
\end{proof}
\begin{proposition}\label{prp:xx}
Suppose that $x$ and $x'$ are the solutions of $F(x)=0$ and $F'(x')=0$ of the smallest norms. For any $\gep_x$ which  satisfies
\begin{equation}\begin{aligned}\label{eq:gep_x}
\gep_x <&
\frac{\Tail_2(\sum_i\mu_i((|x'_i| + \gep_x)_i)\eta_i(|x'_i| +\gep_x))}{\sigma_{g-1}(\cJ_F'\restriction_0)-\gep_{\cJ_F}}+\frac{\sum_i3|c'_i|\cdot|x'_i|^2+2|c'_i|^2|x_i'|^3}{\sigma_{g-1}(\cJ_F'\restriction_0)}\\
+&\frac{\gep_{\cJ_F}}{\sigma_{g-1}(\cJ_{F'}\restriction_0)(\sigma_{g-1}(\cJ_{F'}\restriction_0)-\gep_{\cJ_F})}+\frac{|((T_j\wedge\bigwedge_iP_i)-(T'_j\wedge\bigwedge_iP'_i))_j|}{\sigma_{g-1}(\cJ_{F'}\restriction_0)-\gep_{\cJ_F}}
\end{aligned}\end{equation}
and $|x'_i|+\gep_x < r_{i,x}$, the magnitude $|x'-x|$ is bounded by $\gep_x$.
\end{proposition}
\begin{proof}
  The first two terms are bounds on tail components which come from \Cref{lma:F} for $F$ and $F'$ respectively (the computation for $F'$ is completely analog to that of
  $F$, so we do not redo it). The last two terms come from \Cref{prp:solcont} when comparing the solutions of
$F(0)=-\cJ_F\restriction_0\cdot x$, and $F'(0)=-\cJ_{F'}\restriction_0\cdot x'$.
\end{proof}
\begin{remark}\label{rmr:grad_descent}
We have several computational oriented remarks regarding \Cref{prp:xx}:
\begin{enumerate}
\item\label{i:third_term} Realistically (i.e., this is what happens in practice), the biggest terms by far in \Cref{eq:gep_x} are the last two.
\item As in the case of $r$ in \Cref{lma:close_to_P} and its utilization in the proof of \Cref{r_x}, the quantity $\gep_x$ appears in both sides of the inequality in \Cref{eq:gep_x}. On the RHS it appears as a higher degree term. For the same reason as in \Cref{lma:close_to_P}, the higher degree term is inconsequential.
\item\label{i:computing_x_prime} One computes $x'$ by solving the linear equation for $\cJ_{F'}$. In theory, this raises two issues:
\begin{enumerate}
\item There is an error in computing $\cJ_{F'}$ given $P'_i, T'_i, R'_i$. This error is identical to the error from $\gep_{\cJ_F}$ introduced in \Cref{lma:F_with_errors}, when one substitutes all the $\gep_{j,i}$ for \verb|E|.
\item\label{i:wedge_bound} One has to control the inequality relating
  $\cJ_{F'},x'$ analog to the one relating $\cJ_F, x$ in \Cref{lma:F}, to ensure
  its validity.
\end{enumerate}
However, in practice, we will see in \Cref{sec:program} that in our cases, the quantities appearing here are several orders of magnitude below the
quantities dominating the inequalities.
\item\label{i:grad_descent} Like any gradient descent method, the method at hand can serve as a basis for refining an existing solution, not merely bounding its accuracy. We believe this is a worthy effort, but it involves a lot of parameter tuning and interaction with the (implicit) gradient descent from \Cref{lma:close_to_P}. Thus, it lies beyond the scope of the current work. That being said, we did run one iteration with hand-tailored parameters for the descent. See \Cref{subsec:testing} \Cref{i:grad_decent_bound}.
\end{enumerate}
\end{remark}
\begin{theorem}[Bounding $\sqrt{\sum_i|P'_i-(\text{the closest point to }P'_i\text{ in }H\cap C)|^2}$]\label{thm:pp}
With $x$ and $x'$ as in \Cref{prp:xx}, we have
{\footnotesize
\[\begin{aligned}
\sum_i&|P'_i - \text{closest point to }P'_i\text{ in }H\cap C|^2 <  \sum_i|v_{i,x}(x_i)-f'_i(x'_i)|^2 + \sum_i|P'_i- f'_i(x'_i)|^2\\
<&\sum_i |(P_i-P'_i)+(x_i-x'_i)T_i+x'(T_i-T'_i)+(x_i^2-{x'_i}^2)R_i+{x'_i}^2(R_i-R'_i)+u_i(x)|^2+\sum_i|P'_i- f'_i(x'_i)|^2\\
<&\left(\gep_{0,i}+\gep_x+|x'|\gep_{1,i}+(2\gep_x(|x'|+\gep_x)|x'|+\gep_x^2)(|R'_i|+\gep_{2,i})+|x'|^2\gep_{2,i}+\kappa_{0,i}(|x_i'|+\gep_x)^3\right)^2+\left(|x'|+|R'_i|\cdot|x'_i|^2\right)^2.
\end{aligned}\]
}
\end{theorem}
\begin{proof}
The first inequality is immediate from the definition, and then we simply expand.
\end{proof}
When computing the quantity above, to bound the estimation error, one has to multiply it by $(1+$\verb|E|$)^N$ where $N$ is the number of machine operations involved in the computation.

Intuitively, in \Cref{lma:F}, we perturb points along $C$, which means that if the $H'\wedge T_{P_i}C$s are small, then the magnitude $|H-H'|$ should be much smaller than the distance bound from \Cref{thm:pp}. Below we make this intuitive view precise.
\begin{theorem}[Bounding $|H-H'|$]\label{thm:hh}
Let $n_{H'}$ be the unit normal to $H'$. Then under the assumption of \Cref{prp:xx}, the magnitude $|H-H'|$ is bounded by
\[
(|x'|+\gep_x)\Big(\sum_i|\langle T'_i,n_{H'}\rangle|+\gep_{1,i}\Big)+\Tail_2\Big( \prod_{i=1}^n(1+(|x_i'|+\gep_x)+|c_i'|(|x_i'|+\gep_x)^2+\kappa_{0,i}'(|x_i'|+\gep_x)^3)\Big).
\]
\end{theorem}
\begin{proof}
Since $H'=\wedge_{i=1}^{n}P_i$, the points $P_1,\ldots, P_{n+1}$ together $n_{H'}$ form a basis to $\CC^{n+1}$. Moreover, since $P_i\perp T_i$,  the coefficient of $P_i$ in the representation of $T_i$ in this basis is trivial, that is to say $|T_i\wedge\bigwedge_{k\neq i}P_i|< |\langle T_i,n_{H'}\rangle|$.
Hence:
\[\begin{aligned}
&|H-H'|=\wedge_i (P_i+x_iT_+x_i^2c_iR_i+u_i(x_i))-\wedge_i P_i|
\\ &< \sum_i |x_i| |\langle T_i,n_{H'}\rangle| + \Tail_2\Big( \prod_{i=1}^n(1+|x_i|+|c_i|\cdot|x_i|+\kappa_{0,i}|x_i|)\Big)\\
&< (|x'|+\gep_x)\Big(\sum_i|\langle T'_i,n_{H'}\rangle|+\gep_{1,i}\Big)+\Tail_2\Big( \prod_{i=1}^n(1+(|x_i'|+\gep_x)+|c_i'|(|x_i'|+\gep_x)^2+\kappa_{0,i}'(|x_i'|+\gep_x)^3)\Big).
\end{aligned}
\]
\end{proof}
We close this section with the following important remark.
\begin{proposition}
Assuming that $\dim\mathrm{pseudo}\ker M_{T_i'}<2$ for all the points $P'_i$ of a certified hyperplane, the corresponding theta characteristic $\theta$ is odd i.e., $\dim H^0(C,\gth)=1$, where $M_{T'}$ is the computed version of the matrix $M_T$ defined in \Cref{dfn:M_TandM_R}.
\end{proposition}
%
\section{Numerically certifying Steiner sets given certified theta hyperplanes}\label{sec:certificationSteiner}
%
In this section we assume we are provided with {\em certified} approximate theta
hyperplanes, which come with putative {\em non-certified} level 2 structure in
the form of half-characteristics; i.e. picking an even two torsion point in
the Jacobian we identify $\mathrm{Pic}^{g-1}$ with $\mathrm{Pic}^0$. Thus,
considering these theta hyperplanes as points in an odd affine $SP_2(2g)$
homogeneous space, we give them linear homogeneous space coordinates by picking
a distinguished $0$ among the points in the even affine space.

Recalling the definition of a Steiner set from \Cref{def:SteinerSet}, we
also recall that a pair of pairs of distinct odd theta characteristic
$\{\{\theta_1,\theta_2\},\{\theta_3, \theta_4\}\}$ is called syzygetic (resp.
azygetic) if $\{\theta_1,\theta_2\},\{\theta_3, \theta_4\}$ are (resp. are not)
in the same Steiner set which happens if and only if
$\gth_1+\gth_2+\gth_3 + \gth_4=2K_C$ (resp. $\neq 2K_C$).
The syzygetic and azygetic quadruples (and thus the Steiner sets) can be built
by means of the half-characteristics coordinates above.

Our approach in certifying the Steiner sets is a combined algebraic combinatorial
approach where we do not attempt to directly certify syzygetic quadruples at all:
\begin{proposition}\label{prp:partition}
Let $S,\ol{S}$ be two subsets of the set of unordered pairs of odd theta characteristics of $C$ such that $\#\ol{S}=2^{2g}-2
$ and any pair of elements in either $\binom{\ol{S}}{2}$ or $\ol{S}\times S$ forms an azygetic quadruple then $S$ is a constituent of a Steiner set.
\end{proposition}
\begin{proof}
Each element of $\ol{S}$ is in a different Steiner set, and each element of $S$ is in a Steiner set where none of the elements of $\ol{S}$ are, but there is only one Steiner set left.
\end{proof}
This approach enables to certify azygetic quadruples, which is done by bounding
quantities {\em away} from $0$. We start by setting some definitions:
\begin{definition}\label{dfn:M_D}
  Fixing a positive integer $k$, let $D$ be an effective divisor supported on
  distinct points. Assume that the points in the support of $D\in(|K_C|^*)^{\deg D}$ are ordered, and their representatives in $H^0(K_C)^*$ are of norm $1$. Let $M_D$ be the $\deg D \times \dim\Sym^k H^0(K_C)$
complex matrix whose entry at the row $r$, the column $c$ is the evaluation of the $r$th point of $D$ at the $c$th degree $k$ monomial on $H^0(K_C)$ in lexicographic order.
\end{definition}
Note that the lower $\dim I_k(C)$ singular values of $M_D$ are trivial.  On the other hand, the complex unitary norm on $\Sym^k H^0(K_C)$ under the monomial basis induces an isomorphism $I_k(C)^\perp\cong H^0(k K_C)$. Moreover, this norm induces the Fubini-Study metric on $|k K_C|$. In what follows, we establish a ``distance-like'' (to an extent) notion between elements of $\Sym^kH^0(K_C)$ and
divisors.
\begin{definition}
Given a polynomial $p\in \Sym^kH^0(K_C)$ of norm $1$, we define the \emph{gap} between the divisor $D$ and the polynomial $p$ as:
\begin{equation}\label{eq:distance}
    \fg(D, p):=|M_D\cdot V_p|=\sqrt{\sum_{\substack{{x\in D} \\ {|x| =1}}} |p(x)|^2},
\end{equation}
where $V_p$ is the vector whose entries are the coefficients of $p$.
\end{definition}
\begin{remark}
The gap between the divisor $D$ in ${|K|^*}^{2k(g-1)}$ and $p$ is bounded from below by the maximum of the Fubini-Study distances between the points of $Z(p)\cap C$ and $D$, and from above by the sum of these distances.
\end{remark}
\begin{proof}
This follows from the definition of the gap $\fg$.
\end{proof}
The most interesting case is $k=2,\deg D_i=g-1$. From here on we will
concentrate on this case unless otherwise stated, and abuse the notation and
write $M_{ij}:=M_{D_i + D_j}$.
\begin{definition}\label{dfn:approxm_syz}
 We say that $D_i, D_j$ are $\gep$-approximately syzygetic (resp. $A$-numerically azygetic) if the $(\dim I_2(C)+1)$-th lower singular value of $M_{ij}$ is $<\gep$ (resp. $>A$).
Let  $D_1,D_2, D_3, D_4$ be effective divisors of degree $2g-2$, supported on distinct points. Assume that the pairs $(D_1, D_2)$, $(D_3, D_4)$ are $\gep$-approximately syzygetic, and the pair $(D_1, D_3)$ is syzygetic.

Let $Q_{13}$ to be a norm 1 vector in the 1-dimensional vector space $\rmspan({\ol{Q}_{13}}^\perp, I_2(C))^\perp$, where orthogonality is in the monomial basis for $\Sym^2H^0(K_C)$. Similarly we take $H^0(2K_C)$ to be the unitary orthogonal complement of $I_2(C)$ in $\Sym^2H^0(K)$, and set
$V:=\rmspan(H^0(2K)\otimes[Q_{13}], I_2(C))^\perp$, where orthogonality is in the monomial basis on $\Sym^4H^0(K_C)$. This is a $\dim H^0(4K_C)-\dim H^0(2K_C)=4(g-1)$ dimensional subspace of $\Sym^4H^0(K)$. Evaluating (a unitary orthogonal
basis of) these quartics on the $4(g-1)$ points of $D_1+D_3$ we get a square $4(g-1)\times 4(g-1)$ matrix. We denote its lowest singular value by $B$.
\end{definition}
\begin{proposition}\label{prp:weaktransitivity}
With the settings above, $(D_2, D_4)$ is
\[
\frac{\sqrt{g-1}(2+2\sqrt{2})}{\sqrt{2}-\sqrt{4g-4}\max(|M_{12}|,|M_{34}|)\frac{\gep}{B}}\max(|M_{12}|,|M_{34}|)\gep
\]
approximately syzygetic.
\end{proposition}
\begin{proof}
For $(i,j)=(1,2),(3,4)$, denote by $Q'_{ij}$ the quadric representing the lower $\dim I_2(C) + 1$ singular vector of $M_{ij}$. Then by definition $Q'_{ij}$ represents a hyperplane in $|2K|$ whose gap from the points of $D_i+D_j$ is at most $\gep$.
If we look at the gap between the sum of the divisors and the quartic $Q'_{12}Q'_{34}$ we have that
\[\begin{aligned}
\fg&(D_1+D_2+D_3+D_4, Q'_{12}Q'_{34})^2=\sum_{x\in D_1+D_2+D_3+D_4}|Q'_{12}(x)|^2\cdot |Q'_{34}(x)|^2\\
=&\sum_{x\in D_1+D_2}|Q'_{12}(x)|^2\cdot |Q'_{34}(x)|^2+\sum _{x\in D_3+D_4}|Q'_{12}(x)|^2\cdot |Q'_{34}(x)|^2\\
<& \,\, \gep^2\deg(D_1+D_2)|M_{34}|^2+\gep^2\deg(D_3+D_4)|M_{12}|^2.
\end{aligned}\]
Similarly, we also have
\[\begin{aligned}
\fg(D_1+D_3, Q'_{12}Q'_{34})^2=\sum_{x\in D_1+D_3}|Q'_{12}(x)|^2\cdot |Q'_{34}(x)|^2
< \,\, \gep^2\deg(D_1)|M_{34}|^2+\gep^2\deg(D_3)|M_{12}|^2.
\end{aligned}\]
Recalling $V$ and $B$ from \Cref{dfn:approxm_syz}, the norm of the projection of $[Q'_{12}Q'_{34}]$ on $V$ is bounded by $\gep_R:=\fg(D_1+D_3, Q'_{12}Q'_{34})/B$.
i.e., there exists a quadric $Q$ and a quartic $R$ so that $Q'_{12}Q'_{34}-QQ_{13} + R\in I_4(C)$. The norm of $R$ is bounded by $e_R$, and $\fg(D_1+D_3, R) < \fg(D_1+D_3, Q'_{12}Q'_{34})$.
This gives the following bound for the gap $D_1+D_2+D_3+D_4$ and $QQ_{13}$:
\[\begin{aligned}
\fg&(D_1+D_2+D_3+D_4, QQ_{13}) < \fg(D_1+D_3, Q'_{12}Q'_{34}) + \fg(D_1+D_2 + D_3 + D_4, Q'_{12}Q'_{34})\\
<&(\sqrt{4g-4}+\sqrt{8g-8})\cdot \max(|M_{12}|,|M_{34}|)\gep=\sqrt{g-1}(2+2\sqrt{2})\max(|M_{12}|,|M_{34}|)\gep.
\end{aligned}\]
We now observe that
\[
  |Q'_{12}|=|Q'_{34}|=1\quad\Rightarrow\quad1\geq|Q'_{12}Q'_{34}|\geq\frac{1}{\sqrt{2}}\quad\Rightarrow\quad\geq|QQ_{13}|\geq\frac{1}{\sqrt{2}}-\gep_R,
\]
and so, since $|Q_{13}|=1$, we have $1\geq|Q|\geq\frac{1}{\sqrt{2}}-\gep_R$ and
the result follows.
\end{proof}
\begin{theorem}\label{thm:approx_steiner}
Let $S,\ol{S}$ be two subsets of the set of unordered pairs of odd theta characteristics of $C$ so that $\#\ol{S}=2^{2g}-1$. Assume that there are distinguished elements $\ol{s}\in\ol{S}$, and $s\in S$ such that
\begin{enumerate}
    \item each pair in $\binom{\ol{S}}{2}$ either intersects in a theta characteristic, or forms an $A$-numerically azygetic quadruple,
    \item for any $s'\in S\ssm\{s\}$, the theta characteristics in $s,s'$ are $\gep$-approximately syzygetic,
    \item for any ${\ol{s}}'\in\ol{S}\ssm\{\ol{s}\}$, the theta characteristics in $s,{\ol{s}}'$ are $A$-numerically azygetic,
    \item $A/e$ is bigger than the bound from \Cref{prp:weaktransitivity}
\end{enumerate}
then $S$ is a constituent of a Steiner set.
\end{theorem}
\begin{proof}
    It follows from \Cref{prp:partition} and \Cref{prp:weaktransitivity}.
\end{proof}
%
\section{Numerically certifying the dimensions of $V_{2,\ga}$ and their intersection}\label{sec:dimcert}
%
In this section, we assume that we are given theta hyperplanes up to bounded accuracy as points in $|K_C|$, together with the structure of Steiner sets, and we show
how to bound from below $\codim V_{2,\ga}$ and  $\dim \cap_{\ga\in JC[2]\ssm\{0\}}V_{2,\ga}$, which are the fundamental quantities of the present work, introduced in \Cref{dfn:PropertiesABC}. Assuming the certification of theta hyperplanes and Steiner sets, these bounds are the only remaining obstacles in the proof of \Cref{thm:main}.

We collect some notation for the following propositions. Recall the definition of a Steiner set $\gS_{\ga}$ for a nonzero 2-torsion point $\ga$ from \Cref{def:SteinerSet}. Assume that the linear forms $l_\theta, l_{\gth+\ga}$ representing theta hyperplanes in the Steiner set $\gS_\ga$ are normalized and known up to accuracy $\gep$ in $L_2$-norm in the monomial basis, and let $q_{\gth, \theta+\ga}=l_\gth\cdot l_{\gth+\ga}$. Fixing a Steiner set $\gS_\ga$, we denote by $M_\ga$ the matrix whose rows are the coordinate vectors of the quartic forms $q^2_{\theta, \theta+\ga}$ in the monomial basis of $\Sym^2\Sym^2H_0(K)$. Let $\{u_{\ga,i}\}_i$ be an orthonormal basis for the kernel of $M_\ga$. Last but not least, recall from \Cref{codimV} that $d_\ga$ denotes the expected codimension of $V_{2,\ga}$ inside $\Sym^2\Sym^2 H^0(K_C)$. We set
\[r :=d_\ga + \dim \Sym^n\Sym^2H^0(K_C+\ga) = \dim I_n(\gD(H^0(K_C+\ga)) + \dim \Sym^nH^0(2K_C).\]
\begin{proposition}\label{prp:one_steiner}
  With the notation above, the following statements hold
    \begin{enumerate}
        \item \label{item:q_theta} The quadratic forms $q_{\theta, \theta+\ga}=l_\theta\cdot l_{\theta+\ga}$ are accurate up to $2\gep + \gep^2$.
        \item \label{item:Malpha} The matrix $M_\ga$ is accurate up to $\gep':=\sqrt{2}\cdot2^{g-1} (2(2\gep + \gep^2)+(2\gep + \gep^2)^2)$ in Frobenius norm.
        \item \label{item:sigma_r} If the $r$th singular value of $M_\ga$ is greater than $\gep'$ then the codimension of $V_{2,\ga}$ is $d_\ga$.
        \item \label{item:u_i}  If the hypothesis of \Cref{item:sigma_r} holds then the basis  $\{u_{\ga,i}\}_i$ is accurate up to $\gep'':=\frac{\sigma_{r+1}+\gep'}{\sigma_r-\gep'}$.
    \end{enumerate}
\end{proposition}
\begin{proof}
To prove \Cref{item:q_theta} note that the error in the norm of $l_\theta\otimes l_{\theta+\ga}$ is bounded by $2\gep\cdot 1+\gep^2$. Thus, when taking the symmetric product without normalizing, the norm is between this quantity and half of it.
\Cref{item:Malpha} follows from bounding the operator norm of $M_\ga$ by its Frobenius norm. By \Cref{lem:codimV}, we already have an inequality in the other direction, so the equality in \Cref{item:sigma_r} holds. Finally, \Cref{item:u_i} follows from \Cref{prp:solcont}.
\end{proof}
\begin{proposition}\label{prp:all_steiner}
    With the notations above, assume that $\{u_{{\ga_k}, i}\}_{i, k}$ is accurate up to $\gep''$ for the Steiner set $\gS_{\ga_k}$ for each $1< k < m$. Assume that the $(\dim \ker\pi^2_{K_C})$th singular value of the matrix whose rows are $\{u_{{\ga_k}, i}\}$ is greater than $\gep''\sqrt{d_\ga \cdot m}$ then $\dim \cap_{k=1}^m V_{2,\ga_k}C$ is at most $\dim \ker\pi^2_{K_C}$. Therefore these spaces are equal.
\end{proposition}
\begin{proof}
    This follows from the definition of singular value.
\end{proof}
%
\section{A small matter of programming: Concluding the proof of the main theorem}\label{sec:program}
%
As explained in the introduction, the object of this section is fourfold.
We start by explaining how one computes the non-certified approximated
hyperplanes and their (approximated and uncertified)
intersection points with the curve.
Next we discuss inherent errors in using numerical linear algebra, which we
treat as an external certification oracle. Then we discuss numerical errors
incurred during the certification. Addressing these sources of errors concludes the proof of the main theorem -- \Cref{thm:main}. Finally, we discuss testing the certification code. Our code and data
are
accessible at \cite{CL24}. For a bird's-eye review of the structure of the certification code and its auxiliary output, see the documentation of
the code.

\subsection{Generating non-certified theta hyperplanes}
As explained in \Cref{sec:main}, our proof hinges on first proving
\Cref{thm:main} for a
specific curve in each genus.
In genus 6 we use a plane model of Wiman's sextic~\cite{Wim6}; the unique
smooth non-hyperelliptic genus 6 curve whose automorphism group is $\mathbf{S}_5$. Its
plane model is given by
\[
x^6+y^6+z^6+(x^2+y^2+z^2)(x^4+y^4+z^4)-12x^2y^2z^2=0.
\]
In genus 7 we use a plane model of the Fricke-Macbeath curve~\cite{Fri,Mac};
the unique genus 7 Hurwitz curve. We use a plane model given in
\cite{Hid}, who attributes it to Bradley Brock:
\[
z^8+7xyz^6+21x^2y^2z^4+35x^3y^3z^2+28x^4y^4+2x^7z+2y^7z=0.
\]
In both cases we use the {\em certified} SageMath package~\cite{BruSijZot} to
compute a Riemann matrix and the {\em uncertified} SageMath
package~\cite{BruGan} to compute the
$2^{g-1}(2^g-1)$ linear forms putatively approximating the theta hyperplanes,
together with putative level 2 structure, as in~\eqref{eq:thetaChar}.
Computing the adjoint ideal, we then normalize the planar curves and present
them as an intersection of quadrics in their respective canonical systems.
Lastly we use Julia homotopy continuation~\cite{BreTim} to approximate the
tangency points of the hyperplanes and the curves in the canonical models.

There are of course no numerical errors in generating $I_2(C)$, which is
comprised of quadrics $Q_k$ with integer coefficients. As for the
putative hyperplanes $H'_i$, and their putative intersection points $P'_{ij}$
with the curve (see \Cref{dfn:stability_notations} for the notation), we
measured the following numerical errors:
\[
H'_i\cdot P'_{ij} < 10^{-10}, \quad \sqrt{\sum_k|{P'_{ij}}^tQ_kP'_{ij}|^2}< 10^{-11},\quad \sqrt{\sum_j|H'_iT'_{ij}|^2}<10^{-6},\quad\sqrt{\sum_j|(\wedge_j P'_{ij})_iT'_{ij}|^2}<7\cdot10^{-8}.
\]
Note that the last two errors --  which are essentially two incarnations of the
same object -- are not part of the input data per se.
These last two errors dominate the other input errors, which in turn dominate
the errors introduced by our SVD computations, which dominate errors introduced
by any other numerical computations (which only introduces errors of the order
of magnitude of {\small \verb|E|}).

\subsection{Errors inherent to using SVD}\label{subsec:num}
Per \verb|LAPACK| documentation, the error in the singular values is the machine precision \verb|E| (for IEEE-754 double \verb|E|=$2^{-53}\sim$\verb|1.1e-16|) times the top computed singular value. When bounding the norm of an operator, the maximal error is exactly this value as well. There are two components in the certification process that we care about more than the singular values. The first one is in \Cref{prp:rcont} where we compute a solution of a linear system. The second one is  \Cref{prp:all_steiner}, which uses the left pseudo kernel of the matrices from \Cref{prp:one_steiner}. See \Cref{rmr:lapack_oracle} for the details of the accuracy bounds in these computations.

\subsection{The numerics of \Cref{sec:certificationThetaHyperplanes}}
For most objects considered in this section, we have three different
incarnations:
\begin{itemize}
    \item The \emph{Platonic incarnation}: the theta hyperplanes $H_i$ and the tangency points.
    \item The points $P_{ij}$, which are true intersection points of $H'_i$ with the curve.
    \item The {\em numerically computed} objects we actually have at our disposal -- given by $P_{ij}'$ and $H'_{i}$.
\end{itemize}
The technical part of \Cref{sec:certificationThetaHyperplanes} is roughly divided into two parts, bounding the differences between the Platonic and the second incarnations, and then the second and the last incarnations.
Although one cannot talk about certification without considering the combination of the two parts, they are somewhat different in flavor. Loosely speaking, in the first part we establish claims whose condition is that $r_{i,x}, \kappa_{0,i},\kappa_{1,i}, c_i,\sigma_{g-1}(\cJ_F)$ are all ``$O(1)$", where $r_{i, x},\kappa_{0,i},\kappa_{1,i}$ are all minima
of {\em short} multiplicative (this includes divisions and taking roots)
expressions in
$|R|,|W_R|,|W_T|,q_m,\sigma_{g-2}(M),\sigma_{g-1}(M_T),\sigma_g(M_R)$. The second
part shows how one computes the error in estimating these quantities. Again, broadly speaking, the claims in the first part would hold if the errors computed in the second part are smaller than the quantities in the first part.

The following theorem gives bounds related to the computed hyperplanes $H_i'$ and the points $P_{ij}'$. The bounds which appear below do not bound all the input objects of the corresponding types. Rather, {\em we only consider a subset} of hyperplanes and the respective points that suffice to conclude the proof of the main theorem. In particular, the number of objects should suffice to get the relevant space dimensions.

\begin{theorem}\label{them:accuracy_sec4}
    The computed hyperplanes $H'_i$ and points $P'_{ij}$ which we used for the certification satisfy the following bounds.
    \begin{itemize}
    \item For any used $i$, we have $\sqrt{\sum_j|P'_{ij} - \text{closest point to }P'_{ij}\text{ in }H_i\cap C|^2}<2\cdot10^{-6}$.
    \item For any used $i$, we have $|H'_i-H_i|<10^{-9}$.
    \end{itemize}
\end{theorem}
\begin{proof}
\Cref{sec:certificationThetaHyperplanes} gives exact bounds on what the relations between ``$O(1)$" and ``small" should be. However, the bounds we get on the numbers involved are so good, that there is no point in performing the meticulous computation. In what follows $x < (y,z)$ means that $x<y$ for the
Wiman curve, and $x<z$ for the Fricke-Macbeath curve:
{\small\[
\begin{aligned}
  \ & \sigma_{g-2}(M'),\sigma_{g-1}(M_T') > .1,\qquad
|M'|,|M_T'| < (50, 10), \qquad
  \sigma_g(M_R') > (.3, .7),
  \\
  \ & \sigma_{g-1}(M'),\sigma_g(M_T') < (10^{-11}, 10^{-10}),\qquad
|\{M'Q_i\}_i| < (10, 5),\qquad
|W_T'|,|W_R'| \in [.01, 10], \qquad
|R'|\in[.2, 5]\\
\ &\gep_1,\gep_2 < (10^{-10}, 10^{-11}),\qquad
\gep_0 < 10^{-13},\qquad
\gep_{\cJ_F} < (10^{-9}, 5\cdot10^{-11}),\qquad
|x'|<3\cdot 10^{-7},\qquad
\sigma_i(\cJ_F')\in[.05,2].
\end{aligned}
\]}
Whereas the maximal errors in computing the singular values of the matrices involved are at most $4.1\cdot 10^{-16}$ for the Fricke-Macbeath curve and $2.4\cdot10^{-15}$ for Wiman's sextic.
One easily verifies that the claim in \Cref{rmr:grad_descent} \Cref{i:third_term} about \Cref{eq:gep_x} holds. Thus, by \Cref{prp:xx} and \Cref{thm:pp}, the quantity $\sqrt{\sum_i|P'_i - \text{closest point to }P'_i\text{ in }H\cap C|^2}$ is bounded by $\max|x'|+\max\epsilon_{\cJ_F}/\sigma_{g-1}(\cJ_F)^2+\max\sqrt{\sum_j|(\wedge_j P_{ij})_iT'_{ij}|^2}/\sigma_{g-1}(\cJ_F)$ + quantities which are several orders of magnitude smaller, which is smaller than $2\cdot10^{-6}$.

Turning our attention to $|H'_i-H_i|$, by \Cref{thm:hh}, the bound is given by
\[
2\cdot10^{-6}\cdot(\sum_i|\langle T'_i,n_{H'}\rangle|+\max(\gep_1,\gep_2,\gep_0))+
\Tail_2\prod_i\eta_i(2\cdot10^{-6}),
\]
where $|\langle T'_i,n_{H'}\rangle|<10^{-6}$. Thus the highest terms in the expression are $2\cdot10^{-6}(g-1)10^{-6}$ from the first term, and $g^2\max(|R|)\cdot(2\cdot10^{-6})^2$
from the latter.
\end{proof}
\subsection{Errors incurred by certifying the Steiner set in \Cref{sec:certificationSteiner}}
The ingredients involved in the certification of the Steiner sets are:
\begin{itemize}
    \item The accuracy of the approximated tangency points, for which we have bounds presented above.
    \item The operator norms of the matrices $M_{ij}$ defined in \Cref{dfn:M_D}.
    \item The {\em maximal} (resp. {\em minimal}) computed $\gep$ (resp. $A$) for $\gep$-approximately syzygetic (resp. $A$-numerically azygetic) pairs of pairs of theta characteristics.
    \item The smallest singular value $B$ defined in \Cref{dfn:approxm_syz}.
\end{itemize}
Note that the operator norms $|M_{ij}|$ incur computational errors which we bound via the bound on the computational errors from the tangency points. Specifically, if a point is known with accuracy $\gep$ then the corresponding line in the matrix $M_{ij}$ is known with accuracy $2\gep+\gep^2$. Thus by \Cref{them:accuracy_sec4} the Frobenius norm of the error in matrices $M_{ij}$s is bounded by $\sqrt{\frac{4(g-1)}{g-1}}\cdot(4\cdot10^{-6}+4\cdot10^{-12})<10^{-5}$. Referring to
the last item, one can bound the error in computing $B$ by tracing the steps in \Cref{dfn:approxm_syz}. This is done by bounding the norm of the product of two quadrics between $1/2$ and exactly the product of the norms, bounding the Frobenius norms, and applying \Cref{prp:kercont}. The reason we do not explicitly trace these steps is that we did similar computations several times before, and the quotient $\gep/B$ we get in practice is much smaller than anything that errors in computing $B$ might contribute.
\begin{proposition}\label{prp:correct_steiner}
For the subset of hyperplanes which pass our tests, the labeling of the Steiner sets obtained from the non-certified level 2 structure is correct.
\end{proposition}
\begin{proof}
Recall that by \Cref{thm:approx_steiner} there are four numerical estimations at play here: the $\dim I_2(C)$-th singular value of $\gep$-approximately syzygetic sets measured exactly with $\gep$, and $B$ from \Cref{dfn:approxm_syz}, the $A$-numerically azygetic sets coefficient, and $|M_{ij}|$. They are connected by comparing $A/\gep$ with
\[
\frac{\sqrt{g-1}(2+2\sqrt{2})}{\sqrt{2}-\sqrt{4g-4}\max(|M_{12}|,|M_{34}|)\frac{\gep}{B}}\max(|M_{12}|,|M_{34}|).
\]
The computed bounds on these coefficients are:
$\gep<10^{-8}, |M_{ij}|<3,B>2\cdot 10^{-4}, A>.2$.
The maximal errors incurred in the singular value computations of $M_{ij}$, (i.e., in all the constants beside $B$) are $<7\cdot10^{-17}$ for the Fricke-Macbeath curve, and $<3\cdot10^{-16}$ for Wiman's sextic. As explained above and in more detail in \Cref{dfn:approxm_syz}, there are several singular value computations involved in computing $B$; the biggest error in these computations is $<2\cdot10^{-14}$ for Wiman's sextic and $<5\cdot10^{-15}$ for the Fricke-Macbeath curve.
\end{proof}
\subsection{Errors incurred by certifying \Cref{sec:dimcert}, and the proof of the main theorem}
Recalling the proof outline given in \Cref{sec:main}, we use the fact that by \Cref{L:C7L2} it suffices to prove the claim for one curve in each genus. In \Cref{lem:codimV} we have established an upper bound on the dimension $\dim V_{2,\ga}$. Under the assumption that this upper bound is an equality, we would have  $V_{2,\ga}\supset\ker\pi^2_{K_C}$, which shows that $\cap_\ga V_{2,\ga}\supset \ker\pi^2_{K_C}$. Thus this gives us a lower bound on the dimension of the intersection. We now complete the proof of the main theorem \Cref{thm:main}.
\begin{proof}[Proof of \Cref{thm:main}]
We first give a lower bound on $\dim V_{2,\ga}$, which turns out to be identical to the upper bound $d_\ga$ from \Cref{codimV}. Next, we give a lower bound on
$\dim \rmspan\cup_\ga V_{2,\ga}^\perp=\dim(\cap_\ga V_{2,\ga})^\perp$. Thus, we get an upper bound on the $\dim(\cap_\ga V_{2,\ga})$, which then coincides with $\dim\ker\pi^2_{K_C}$, which is the lower bound we have for the dimension of this intersection. By \Cref{prp:correct_steiner} our Steiner sets are correct, so we can compute the lower bound using the hyperplanes we certified. The accuracy of each such hyperplane is bounded by $10^{-9}$. Hence, the accuracy of each line of the matrix representing the Steiner set in $\Sym^2\Sym^2H^0(K_C)$, where each line is the $\Sym^2$ of the product of hyperplanes, is bounded by $4\cdot10^{-9}$ + $O(10^{-18})$. Ignoring terms
of magnitude $10^{-20}$, the matrices in question have at most $2^{g-2}(2^{g-1}-1)<2^{2g-3}$ lines, and so the Frobenius norm of the error in computing these matrices is bounded by $2^{2g-3}\cdot 4\cdot 10^{-9}$. To begin with, this gives us a bound on the error computed singular values, which are
\[\begin{aligned}
\ &\text{Fricke-Macbeath curve:}&
\max\sigma_1 = 6.02764,\quad
\min\sigma_{361}= 0.0379421,\quad
\max\sigma_{362} = 1.03606\cdot 10^{-13}\\
\ &\text{Wiman's sextic :}&
\max\sigma_1= 7.78234,\quad
\min\sigma_{181} = 0.00833475,\quad
\max\sigma_{182} = 6.46229\cdot10^{-14}
\end{aligned}\]
where the error incurred in this computation in these cases are
$6.7\cdot10^{-16},8.7\cdot10^{-16}$ respectively.

We have now proven the dimension claim for the $V_{2,\ga}$.
Next, by \Cref{prp:kercont}, the error in the bases of the kernels is bounded by
$\frac{|1.04\cdot10^{-13}-8.2\cdot10^{-6}|}{.04-8.2\cdot10^{-6}}<2.1\cdot10^{-4}$ for the Fricke-Macbeath curve, and $\frac{|6.5\cdot10^{-14}-2.1\cdot10^{-6}|}{.01-2.1\cdot10^{-6}}<2.2\cdot10^{-4}$ for Wiman's sextic.

For the Fricke-Macbeath curve, we intersect only $96$ many $V_{2,\ga}$s, whereas for Wiman's sextic, we intersect only $56$ of them. In both cases, this means the Frobenius norm of the error is bounded by $\sqrt{100}\cdot 2.2\cdot 10^{-4}< 2.5\cdot 10^{-3}$.
The singular values we compute for the corresponding matrices are
\[\begin{aligned}
\ &\text{Fricke-Macbeath curve:}&
\sigma_1 = 7.26053,\quad\sigma_{171}  = 1.17418,\quad\sigma_{172} = 1.86973\cdot10^{-12}\\
\ &\text{Wiman's sextic :}&
\sigma_1 = 9.21738,\quad\sigma_{120}  = 1.07752,\quad\sigma_{121} = 4.10144\cdot10^{-12}
\end{aligned}\]
where the error incurred in this computation in these cases are
$8.1\cdot10^{-16},1.1\cdot10^{-15}$ respectively.
\end{proof}
\subsection{Testing}\label{subsec:testing}
Although testing cannot prove that a program is correct, it is an important part of any programming endeavor, as it may provide a heuristic assurance for the correctness of the code.
First, it is important to note that the entire point of \Cref{sec:certificationThetaHyperplanes} and \Cref{sec:certificationSteiner} is to certify the computations which were done in a completely different context. This full separation between the way the input that was generated and our certification code is an important test in itself.

Next, there are two types of numerical tests we perform. The first of them is verifying known facts, like the essential dimensions appearing in the matrices $M', M'_T, M'_R$ being $g-2, g-1, g$, or the dimensions in \Cref{sec:dimcert} being smaller than the dimension bounds given in \Cref{sec:main} (on the other hand, an inequality of these numbers in the reverse direction would have just meant that our theorem is wrong). Another verification is via the fact that the point-hyperplane (approximate) incidence relation is invariant under the symmetry group action.
The second type of tests are incarnations of the rule of thumb ``numbers behave randomly unless they have a known reason not to". The incarnations of this rule in our code are:
\begin{enumerate}
    \item The spectral gaps on the essential dimensions of the spans of the rows or columns of $M', M'_T, M'_R$, or the matrices checking azygetic quadruples, or the matrices certifying the computations in \Cref{sec:certificationSteiner} are big. That is to say: if we denote this dimension by $d_{\text{ess}}$ for each of the cases then in these cases $\log_{10}\sigma_{d_{\text{ess}}+1}$ are bounded from above by $-10,-12,-11,-12$ respectively, whereas $\sigma_{d_\text{ess}}$ is bounded from below by $.3,.01,0.00833475,1.07752$ respectively, and $\sigma_1$ is bounded from above by $50,10,7.78234,9.21738$ respectively. This is in contrast to the case for the spectral gap for checking syzygetic quadruples, where $d_{\text{ess}} = \dim I_2(C) - 1$ and $\sigma_{d_{\text{ess}}+1}$ is bounded from above merely by $10^{-8}$, which is far bigger a bound than the corresponding bound for azygetic quadruples. The fact that both this bound and the bound on $\sqrt{\sum_i |H'_k\wedge T'_{ki}|^2}$ are relatively big reflects the fact that the distance of $\sum_jP'_{ij}$ from being half canonical is big, due to the way the $P'_{ij}$s are computed -- see the discussion above \Cref{thm:hh}.
    \item Points and hyperplanes are in practice far off from one another
      in the Fubini Study metrics: $10^{-5}$ away for hyperplanes, whereas for
      points the distance is at least
      $10^{-3}$ for Wiman's sextic and  $5\cdot10^{-4}$ for the Fricke-Macbeath
      curve.
\item\label{i:grad_decent_bound} One step of the iterative gradient descent (via inverse function theorem) algorithm alluded to in \Cref{rmr:grad_descent} \Cref{i:grad_descent} is performed, and it improves the solution by at least a factor of $0.85$.
\end{enumerate}

We close the article with a question that concerns the next genera from the perspective of explicit computations:

\begin{question}
To implement our certification procedure for genus 8 curves, one needs both a general curve and good numerical behavior—such as widely-spread theta hyperplanes in the ambient space and distinct tangency points. Are there explicit examples of Brill–Noether general, non-hyperelliptic genus 8 curves with sufficiently large automorphism group to support both construction and computation? See~\cite{Muk}, \cite{JP23}, and the Macaulay2 package~\cite{GraSch} for potential directions.    
\end{question}

\noindent {\bf Acknowledgements.} We thank the anonymous referee for constructive feedback, including valuable mathematical input that improved the content and presentation of this work.

\bibliographystyle{alpha}
\bibliography{main}

\begin{thebibliography}{ACGH85}

\bibitem[ACGH85]{ACGH}
Enrico Arbarello, Maurizio Cornalba, Phillip Griffiths, and Joe Harris.
\newblock {\em
  \href{https://link.springer.com/book/10.1007/978-1-4757-5323-3}{Geometry of
  {A}lgebraic {C}urves: Volume I}}.
\newblock Springer-Verlag, New York, 1985.

\bibitem[Aro64]{Aro}
Siegfried~Heinrich Aronhold.
\newblock
  \href{https://www.biodiversitylibrary.org/item/109317#page/517/}{{\"{U}}ber
  den gegenseitigen {Z}usamemmenhang der 28 {D}oppeltangenten {e}iner
  allgemeiner {C}urve 4ten {G}rades}.
\newblock {\em Monatberichter der Akademie der Wissenschaften zu Berlin}, pages
  499--523, 1864.

\bibitem[BT17]{BreTim}
Paul Breiding and Sascha Timme.
\newblock
  \href{https://doi.org/10.1007/978-3-319-96418-8_54}{HomotopyContinuation.jl:
  A Package for Homotopy Continuation in Julia}.
\newblock In {\em International Congress on Mathematical Software}, 2017.

\bibitem[CEFS13]{ChiEisFarSch}
Alessandro Chiodo, David Eisenbud, Gavril Farkas, and Frank-Olaf Schreyer.
\newblock
  \href{https://link.springer.com/article/10.1007/s00222-012-0441-0}{{Syzygies
  of torsion bundles and the geometry of the level $l$ modular variety over
  $\overline{\mathcal{M}}_{g}$}}.
\newblock {\em Inventiones mathematicae}, 194(1):73--118, 2013.

\bibitem[{\c{C}}KRS19]{CelKulRenNamGenus4}
T\"urk\"u~\"Ozl\"um {\c{C}}elik, Avinash Kulkarni, Yue Ren, and Mahsa {Sayyary
  Namin}.
\newblock \href{https://doi.org/10.1016/j.jalgebra.2019.07.037}{Tritangents and
  their space sextics}.
\newblock {\em Journal of Algebra}, 538:290--311, 2019.

\bibitem[{\c{C}}L24]{CL24}
T\"urk\"u~\"Ozl\"um {\c{C}}elik and David Lehavi.
\newblock \href{https://arxiv.org/src/2401.02235/anc}{Supplementary Software
  and Data Repository for the Current Work},
  \texttt{https://arxiv.org/src/2401.02235/anc}, 2024.

\bibitem[Cob29]{Cob}
Arthur~B. Coble.
\newblock {\em
  \href{https://archive.org/details/in.ernet.dli.2015.211563}{Algebraic
  Geometry and Theta Functions}}.
\newblock American Mathematical Society, Colloquium Publications Series, 1929.

\bibitem[CS03]{CapSer}
Lucia Caporaso and Edoardo Sernesi.
\newblock \href{https://doi.org/10.1515/crll.2003.070}{Characterizing curves by
  their odd theta-characteristics}.
\newblock {\em Journal für die reine und angewandte Mathematik},
  2003(562):101--135, 2003.

\bibitem[Deb85]{Debarre85}
Olivier Debarre.
\newblock {\em
  \href{http://uat-library.msri.org/books/Book28/files/debarre.pdf}{The
  Schottky problem}}, pages 57--64.
\newblock MSRI Publications 28. Cambridge University Press, 1985.

\bibitem[Dol03]{dolgachev_2003}
Igor Dolgachev.
\newblock {\em
  \href{https://www.cambridge.org/core/books/lectures-on-invariant-theory/9E1B186438B3F778680C4E7E0BCD3D1A}{Lectures
  on Invariant Theory}}.
\newblock London Mathematical Society Lecture Note Series. Cambridge University
  Press, 2003.

\bibitem[Dol12]{Dol}
Igor~V. Dolgachev.
\newblock {\em
  \href{https://www.cambridge.org/core/books/classical-algebraic-geometry/4372BC30A2D9BADB93FDC71ACBDEEBC6}{Classical
  Algebraic Geometry: A Modern View}}.
\newblock Cambridge University Press, 2012.

\bibitem[Don88]{Donagi88}
Ron Donagi.
\newblock \href{https://doi.org/10.1007/BFb0082807}{The schottky problem}.
\newblock In Edoardo Sernesi, editor, {\em Theory of Moduli}, pages 84--137.
  Springer, 1988.

\bibitem[Fri99]{Fri}
Robert Fricke.
\newblock \href{https://link.springer.com/article/10.1007/BF01476163}{Ueber
  eine einfache Gruppe von 504 Operationen}.
\newblock {\em Mathematische Annalen}, 52(2–3):321–339, 1899.

\bibitem[GB21]{BruGan}
Sohrab Ganjian and Nils Bruin.
\newblock {\em \href{https://github.com/nbruin/RiemannTheta}{RiemannTheta: A
  SageMath package for evaluating {R}iemann theta functions}}, 2021.

\bibitem[Gru12]{Gru12}
Samuel Grushevsky.
\newblock {\em
  \href{https://library2.msri.org/books/Book59/files/55gru.pdf}{The Schottky
  problem}}, pages 129--164.
\newblock MSRI Publications 59. Cambridge University Press, 2012.

\bibitem[Gu{\`a}02]{Gua2002}
Jordi Gu{\`a}rdia.
\newblock \href{https://doi.org/10.1016/S0021-8693(02)00049-2}{Jacobian
  {N}ullwerte and algebraic equations}.
\newblock {\em Journal of Algebra}, 253:112--132, 2002.

\bibitem[Hid18]{Hid}
Ruben~A. Hidalgo.
\newblock \href{https://doi.org/10.1007/s40879-017-0166-0}{About the
  Fricke–Macbeath curve}.
\newblock {\em European Journal of Mathematics}, 4(1):313--325, 2018.

\bibitem[HPS24]{hanselman2024equationsgenus4curves}
Jeroen Hanselman, Andreas Pieper, and Sam Schiavone.
\newblock \href{https://arxiv.org/abs/2402.03160}{Equations of genus $4$ curves
  from their theta constants}, 2024.

\bibitem[Kri10]{Krichever10}
Igor Krichever.
\newblock
  \href{https://annals.math.princeton.edu/wp-content/uploads/annals-v172-n1-p09-p.pdf}{Characterizing
  Jacobians via trisecants of the Kummer variety}.
\newblock {\em Annals of Math.}, 172(1):485--516, 2010.

\bibitem[KS13]{KS2013}
Igor Krichever and Takahiro Shiota.
\newblock \href{https://arxiv.org/abs/1111.0164}{Soliton equations and the
  Riemann-Schottky problem}.
\newblock In Gavril Farkas and Ian Morrisob, editors, {\em Handbook of Moduli
  Volume II}, volume~25 of {\em Advanced Lectures in Mathematics}, pages
  205--258. International Press, 2013.

\bibitem[Leh15]{LehGenus4}
David Lehavi.
\newblock \href{https://doi.org/10.1093/imrn/rnu235}{Effective Reconstruction
  of Generic Genus 4 Curves from Their Theta Hyperplanes}.
\newblock {\em International Mathematics Research Notices},
  2015(19):9472--9485, 2015.

\bibitem[Leh24]{LehGenus5}
David Lehavi.
\newblock \href{https://doi.org/10.1080/10586458.2022.2041133}{Effective
  Reconstruction of Generic Genus 5 Curves from their Theta Hyperplanes}.
\newblock {\em Experimental Mathematics}, 33(1):150--164, 2024.

\bibitem[LS96]{LaSe}
H.~Lange and E~Sernesi.
\newblock Quadrics containing a {P}rym-canonical curve.
\newblock {\em Journal of Algebraic Geometry}, 5(2):387--399, 1996.

\bibitem[Mac65]{Mac}
Alexander~Murray Macbeath.
\newblock
  \href{https://londmathsoc.onlinelibrary.wiley.com/doi/abs/10.1112/plms/s3-15.1.527}{On
  a {C}urve of {G}enus 7}.
\newblock {\em Proceedings of the London Mathematical Society}, 15(1):527--542,
  1965.

\bibitem[Muk88]{Muk}
Shigeru Mukai.
\newblock \href{https://doi.org/10.1016/B978-0-12-348031-6.50026-7}{Curves,
  {K}3 Surfaces and Fano $3$-folds of Genus $\leq$ 10}.
\newblock In {\em Algebraic Geometry and Commutative Algebra}, pages 357--377.
  Academic Press, 1988.

\bibitem[Mum07]{Mum}
David Mumford.
\newblock {\em Tata Lectures on Theta I (reprint of 1983 edition)}.
\newblock Birkhäuser, Boston, 2007.

\bibitem[NBZ19]{BruSijZot}
Jeroen~Sijsling Nils~Bruin and Alexandre Zotine.
\newblock \href{https://msp.org/obs/2019/2-1/p10.xhtml}{Numerical computation
  of endomorphism rings of {J}acobians}.
\newblock In {\em ANTS XIII: Proceedings of the Thirteenth Algorithmic Number
  Theory Symposium, The Open Book Series 2-1}, pages 155--171, 2019.

\bibitem[Pau03]{JP23}
Jennifer Paulhus.
\newblock {\href{https://doi.org/10.1007/978-3-031-17859-7_35}{A Database of
  Group Actions on Riemann Surfaces}}.
\newblock In I.~Chen X. Katzarkov L. Park~J. Cheltsov, editor, {\em Birational
  Geometry, Kähler–Einstein Metrics and Degenerations}, pages 693--708.
  Springer, Cham, 2003.

\bibitem[Sch]{GraSch}
Frank-Olaf Schreyer.
\newblock {\em Macaluay2 software package:
  {\href{https://www.math.uni-sb.de/ag/schreyer/home/M2/doc/Macaulay2/KoszulDivisorOnPic14M8/html/index.html}{KoszulDivisorOnPic14M8
  -- Construction of random curves of degree 14 and genus 8}}}.

\bibitem[Shi86]{Shiota86}
Takahiro Shiota.
\newblock \href{https://doi.org/10.1007/BF01388967}{Characterization of
  Jacobian varieties in terms of soliton equations}.
\newblock {\em Inventiones mathematicae}, 83(2):333--382, 1986.

\bibitem[Shi89]{Shiota89}
Takahiro Shiota.
\newblock The kp equation and the schottky problem.
\newblock {\em Sugaku}, 41(1):16--33, 1989.
\newblock Translated in Sugaku Expositions 3 (1990), no. 2, 183–211.

\bibitem[Wim96]{Wim6}
Anders Wiman.
\newblock \href{http://eudml.org/doc/157818}{Zur Theorie der endlichen Gruppen
  von birationalen Transformationen in der Ebene}.
\newblock {\em Mathematische Annalen}, 48(1):195--240, 1896.

\end{thebibliography}
\appendix
%
\section{Computing theta hyperplanes numerically}\label{sec:computingThetaHyerplanes}
%
Here we recall the Riemann theta function and its connections to theta hyperplanes. Our primary reference for the background material is \cite{Mum}.
Let $C$ be a complex smooth algebraic curve. Let $\omega_1,\dots , \omega_g$ be a basis of $H^0(K_C)$, and let $\ga_1,\dots , \ga_g, \gb_1, \dots, \gb_g$ be a symplectic basis of $H_1(C,\ZZ$). The following $g\times 2g$ matrix is called the \emph{period matrix} of the curve (with the given choice of basis):
\begin{equation}\label{eq:periodMatrix}
(\tau_1|\tau_2):=\left(\left(\int_{\ga_i} \omega_j\right)\Big|\left(\int_{\gb_i} \omega_j\right)\right),
\end{equation}
whereas $\tau:=\tau_1^{-1}\tau_2$ is called a \emph{Riemann matrix} of the algebraic curve $C$. The Riemann matrix lies in Siegel upper half space $\HH_g$ -- the space of $g \times g$ symmetric matrices with positive definite imaginary part.
The theta function with characteristic $\varepsilon,\delta \in \{ 0,1\}^g$ is a
complex-valued function defined on $\CC^g\times \HH_g$:
 \[
\theta\begin{bmatrix} \varepsilon \\ \delta \end{bmatrix}({\bf z}\, |\, \tau)\,\,\, = \,\,\,
\sum_{{\bf n} \in \ZZ^g} {\rm exp} \left( \pi i \left({\bf n}+\frac{\varepsilon}{2}\right)^T \tau \left({\bf n}+\frac{\varepsilon}{2}\right) + \left({\bf n} + \frac{\varepsilon}{2}\right)^T  \left({\bf z} + \frac{\delta}{2}\right) \right).
\]
When $\varepsilon=\delta={\bf 0}$, this is nothing but the \emph{Riemann theta function}, which differs by an exponential factor from the latter. The characteristic is called even, odd if $\varepsilon\cdot \delta\equiv 0,1 \pmod 2$ respectively. So there are $2^{g-1}(2^g-1)$ odd and $2^{g-1}(2^g+1)$ even characteristics. For fixed $\tau$, the values $\theta\begin{bmatrix} \varepsilon \\ \delta \end{bmatrix}({\bf 0}\, |\, \tau)$ at ${\bf z}= {\bf 0}$ are known as \emph{theta constants}. We will also call a theta constant the evaluation of the first-order derivatives of the theta function at ${\bf 0}$:
\[
  \theta^{\varepsilon,\delta}_{i}:=\frac{\partial}{\partial z_i}\theta\begin{bmatrix} \varepsilon \\ \delta \end{bmatrix}({\bf z}\, |\, \tau)_{\big|_{{\bf z} ={\bf 0}}}.
\]
The following statement is classical, which follows from the Riemann singularity theorem. It formulates the theta characteristic divisors in terms of the canonical model (see \cite[Page 228]{ACGH} or \cite[Theorem 2.2]{Gua2002}).
\begin{theorem}
     Let $\tau_1$ be the first $g \times g$ part of the period matrix~\eqref{eq:periodMatrix}. Let $\theta$ be an effective theta characteristic divisor of degree $g-1$ with $\dim H^0(\theta)=1$. The equations of the theta hyperplanes $H_\theta$ arising from $\theta$ on the canonical model are given as follows:
\begin{equation}\label{eq:thetaChar}
    \begin{pmatrix}
    \theta^{\varepsilon,\delta}_1 & ,\dots , & \theta^{\varepsilon,\delta}_g
    \end{pmatrix} \cdot \tau_1^{-1} \cdot \begin{pmatrix}
    X_1 \\ \vdots \\ X_g
    \end{pmatrix}
   = 0,
    \end{equation} where the characteristic $\begin{bmatrix} \varepsilon \\ \delta  \end{bmatrix}$
    ranges over the odd ones.
\end{theorem}
%
\section{Some results in computational linear algebra}\label{sec:linear}
%
Unless otherwise specified, we denote the singular values of a matrix by
$\sigma_i$ which will be assumed to be in decreasing order. In what follows, we
use the operator norm and the Frobenius norm for the matrices, denoted by
$|\cdot|$ and $||\cdot||_{F}$ respectively. The same notation $|\cdot|$ will be
used also for the Euclidean norm, which will be distinguishable from the
objects.
\begin{theorem}\label{thm:sing_sum}
     Let $A,E$, be matrices of the same dimension, then $|\sigma_i(A+E)-\sigma_i(A)|\leq\sigma_1(E)=|E|$.
\end{theorem}
\begin{proposition}[Pseudo kernels are continuous]\label{prp:kercont}
  Let $A, E$ be matrices of the same dimensions. Assume that the projection of
  a vector $v$ on the span of the first $k-1$ singular vectors of $A$ is trivial, then its projection on the first $k-1$ singular vectors of $A+E$ has norm at most  $\frac{\sigma_k(A) + |E|}{\sigma_{k-1}(A)-|E|}|v|$.
\end{proposition}
\begin{proof}
  We have $|(A+E)v| \leq(\sigma_k(A) + |E|)|v|$. Since for $i < k $ we have
  $\sigma_i(A+E)\geq\sigma_i(A)-|E|\geq\sigma_{k-1}(A)-|E|$, the statement follows.
\end{proof}
\begin{proposition}[Solutions of linear equations are continuous]\label{prp:solcont}
Suppose that $A$ is an invertible matrix and $Ax=y$, and $(A+E)x'=y+\delta_y$, where the smallest singular value of $A$ satisfies $\sigma_r(A)>a$. Then $|x'-x|<(|E|/a+|\delta_y|)/(a-|E|)$.
\end{proposition}
\begin{proof} Setting $B=A+E$, we have
$A^{-1}-B^{-1}=B^{-1}BA^{-1}-B^{-1}AA^{-1}=B^{-1}(B-A)A^{-1}$. Whence
\[
A^{-1}y -(A+E)^{-1}(y+\gd_y) = (A+E)^{-1}EA^{-1}y - (A+E)^{-1}\delta_y.
\]
This implies that
$|A^{-1}y -(A+E)^{-1}(y+\gd_y)|\leq \frac{|E|\cdot|y|}{a(a-|E|)}+\frac{|\gd_y|}{a-|E|}.$
\end{proof}
\begin{remark}[Using LAPACK as a certification accuracy oracle]\label{rmr:lapack_oracle}
For our computations, we use the linear algebra library \verb|LAPACK|.
Denote by \verb|E| the machine numerical epsilon. We trust \verb|LAPACK| accuracy bounds: \verb|SERRBD|$(A)$, \verb|VERRBD|$(A,{\verb|i|})$, \verb|UERRBD|$(A,{\verb|i|})$, which denote the errors in singular values of the matrix $A$, and in the angles of singular vectors for the matrix $A$ respectively. The values and the vectors are computed by the function \verb|zgesvd|.
The values of \verb|VERRBD| and \verb|UERRBD| are computed by the function \verb|ddisna|, whereas \verb|SERRBD| is the top computed singular value times \verb|E|. We note that if several consecutive singular values are close then the values given by \verb|ddisna| are meaningless.
Finally, we also trust the least square solution function \verb|zgelss| in the sense that the error is the reported error in the over-defined case (as norm of
the last entries of the returned values). That being mentioned, the reported error is the square root of a sum of numbers, so we have to account for the error in computing this sum as well.
\end{remark}
We continue with bounding the singular values of machine-computed matrices.
\begin{proposition}\label{prp:computed_sing}
  Let $A$ be an $n\times m$ matrix. If $A'$ is the computed representation
  of $A$ then
\[|\sigma'_i(A')-\sigma_i(A)|\leq \sigma'_1(A')\cdot \,{\verb|E|} + \sqrt{\sum_j \gep_{\mathrm{row}_j}^2} \cdot(1+ (2n + 1){\verb|E|}),\]
where $\gep_{\mathrm{row}_j}$ is a bound on the $L_2$-norm of the $j$th row of $A'-A$.
\end{proposition}
\begin{proof}
Note that $|\sigma'_i(A')-\sigma_i(A)|\leq |\sigma'_i(A')-\sigma_i(A')|+|\sigma_i(A')-\sigma_i(A)|$.
Following the \verb|LAPACK| documentation, the first summand is bounded by $\sigma'_1(A')\,\cdot$ {\verb|E|}. By the definition of the Frobenius form and \Cref{thm:sing_sum} the second summand is bounded by the second summand in proposition statement.
\end{proof}
\begin{definition}\label{dfn:abovebelow}
    We denote the bound on the singular values of $A$ above and below by $\sigma_i(A)^+$ and $\sigma_i(A)^-$ respectively.
\end{definition}
\begin{remark}[Ignoring the error in estimating the error]
We will compute our error estimates using a computer. The acute reader will notice that this computation typically involves adding and sometimes multiplying positive real numbers. Thus this computation also introduces errors. However, since we are talking here about positive real numbers the error introduced in the computation will have a size of the error time $1+O({\verb|E|})$. This extra error is meaningless and will be ignored in the presentation.
\end{remark}

\end{document}